\documentclass[leqno,11pt]{amsart}
\usepackage{amssymb, amsmath}
\usepackage{amsthm, amsfonts,mathrsfs}

\usepackage{color}

\newcommand\NN{{\mathbb N}}
\newcommand\RR{{{\mathbb R}}}

\newcommand\cA{{\mathcal A}}

\newcommand\cC{{\mathcal C}}
\newcommand\cD{{\mathcal D}}

\newcommand\cE{{\mathcal E}}

\newcommand\cL{{\mathcal L}}

\newcommand\cN{{\mathcal N}}

\newcommand\cX{{\mathcal X}}

\def\p{\partial}
\newcommand\pP{{\bf P}}
\newcommand\iI{{\bf I}}

\newcommand\eps{{\varepsilon}}
\newcommand\grad{{\nabla_{\!x}}}

\newtheorem{theo}{Theorem}[section]
\newtheorem{lemm}[theo]{Lemma}

\newtheorem{theorem}{Theorem}[section]
\newtheorem{lemma}[theorem]{Lemma}
\newtheorem{proposition}[theorem]{Proposition}

\numberwithin{equation}{section}

\begin{document}

\title[]
{Incompressible Navier-Stokes-Fourier Limit from The Boltzmann Equation: Classical Solutions}

\author{Ning Jiang, Chao-Jiang Xu  \& Huijiang Zhao}
\date{\today}

\address{\noindent \textsc{N. Jiang \\
Mathmatical Sciences Center, Tsinghua University 100084, Beijing, P.R. China}}
\email{njiang@tsinghua.edu.cn}
\address{\noindent \textsc{C.-J. Xu \\
Universit\'e de Rouen, CNRS UMR 6085, Laboratoire de Math\'ematiques, 76801 Saint-Etienne du Rouvray, France\\
and\\
School of Mathematics, Wuhan university 430072, Wuhan, P.R. China}}
\email{Chao-Jiang.Xu@univ-rouen.fr}
\address{\noindent \textsc{H.-J. Zhao \\
School of Mathematics, Wuhan university 430072, Wuhan, P.R. China}}
\email{hhjjzhao@whu.edu.cn}


\begin{abstract}
The global classical solution to the incompressible Navier-Stokes-Fourier equation with small initial data in the whole space is constructed through a zero Knudsen number limit from the solutions to the Boltzmann equation with general collision kernels. The key point is the uniform estimate of the Sobolev norm on the global solutions to the Boltzmann equation.
\end{abstract}

\maketitle

\section{Introduction}
\subsection{Boltzmann Equation}
We consider the following Boltzmann equation in the incompressible Navier-Stokes scaling,
\begin{equation}\label{e1}
\begin{cases}
\partial_tf_\varepsilon+\frac{1}{\varepsilon} v\cdot\nabla_{x}f_\varepsilon=\frac{1}{\varepsilon^2}
\mathcal{Q}(f_\varepsilon, f_\varepsilon),\\
f_\varepsilon|_{t=0}=f_{\varepsilon, 0},
\end{cases}
\end{equation}
where $\eps$ denotes the Knudsen number, which is the ratio of the mean free path and the macroscopic length scale. Here $f(t,x,v)$ is the density distribution function of particles, having position $x\in \mathbb{R}^3$ and velocity $v\in \mathbb{R}^3$ at time $t \geq 0$. The right-hand side of \eqref{e1} is the Boltzmann bilinear collision operator, which is given in the classical $\sigma\mbox{-}$representation by
\[
\mathcal{Q}(g, f)=\int_{\RR^3}\int_{\mathbb S^{2}}B\left({v-v_*},\sigma
\right)
 \left\{g'_* f'-g_*f\right\}\,\mathrm{d}\sigma\mathrm{d}v_*\,,
\]
which is well-defined for suitable functions $f$ and $g$ specified later. In above expression, $f'_*=f(t,x,v'_*), f'=f(t,x,v'), f_*=f(t,x,v_*), f=f(t,x,v)$, and for $\sigma \in \mathbb{S}^2$,
\begin{equation}\nonumber
v'=\frac{v+v_*}{2}+ \frac{|v-v_*|}{2}\sigma\,, \quad v'_*= \frac{v+v_*}{2}- \frac{|v-v_*|}{2}\sigma\,,
\end{equation}
which gives the relation between the post and pre collisional velocities that follow from the conservation of momentum and kinetic energy.

For monatomic gas, the non-negative cross-section $B(z,\sigma)$ depends only on $|z|$ and the scalar product $\frac{z}{|z|}\cdot \sigma$. We assume that it takes the form
\begin{equation}\label{kernel form}
B(v-v_*, \cos \theta)= |v-v_*|^\gamma b(\cos \theta)\,, \quad \cos \theta= \frac{v-v_*}{|v-v_*|}\cdot \sigma\,, \quad 0\leq \theta \leq \frac{\pi}{2}\,,
\end{equation}
where $\gamma > -3$ is the index of the kinetic factor. For the angular factor $b(\cos \theta)$, we consider two cases:

\begin{itemize}
\item The non-cutoff case, $b(\cos \theta)$ behaves like
\begin{equation}\label{non-cutoff kernel}
b(\cos \theta)\sim K  \theta^{-2-2s}\,,\quad \mbox{when}\quad\! \theta\rightarrow 0^+\,,
\end{equation}
for some constants $K>0$ and $0< s < 1$.
\item The Grad angular cutoff case, $b(\cos \theta)$ satisfies
\begin{equation}\label{Grad}
\int^{\frac{\pi}{2}}_0 b(\cos \theta)\sin \theta\,\mathrm{d}\theta < \infty\,.
\end{equation}
\end{itemize}

We consider the fluctuation around a renormalized Maxwellian distribution
\begin{equation}\nonumber
\mu=(2\pi)^{-\frac{3}{2}}\exp\left(-\tfrac{|v|^2}{2}\right)\,,
\end{equation}
by setting $f_\varepsilon(t,x,v)=\mu+\varepsilon\sqrt{\mu}g_\varepsilon (t,x,v)$, and
\begin{equation}\nonumber
  \Gamma(g,h)=\mu^{-1/2}\mathcal{Q}(\sqrt{\mu}g, \sqrt{u}h)\,,
\end{equation}
the linearized Boltzmann operator $\mathcal{L}$ takes the form
\begin{equation}\nonumber
 \mathcal{L}g=-\Gamma(\sqrt{\mu}, g)-\Gamma(g, \sqrt{\mu})\,.
\end{equation}

Now the original problem \eqref{e1} is reduced to the Cauchy problem for the fluctuation $g_\eps$
\begin{equation}\label{e2}
\begin{cases}
\partial_tg_\varepsilon+\frac 1 \varepsilon v\cdot\nabla_{x}g_\varepsilon+\frac 1{\varepsilon^2}\mathcal{L}g_\varepsilon=\frac{1}{\varepsilon}
\Gamma(g_\varepsilon, g_\varepsilon),\\
g_\varepsilon|_{t=0}=g_{\varepsilon, 0}\,,
\end{cases}
\end{equation}
where $g_{\eps,0}$ is give by $f_{\eps,0}(x,v)= \mu + \eps \sqrt{\mu} g_{\eps,0}(x,v)$.

\subsection{Notations} Before we state our main theorems, we introduce some notations. It is well known that the null space $\mathcal{N}$ of $\mathcal{L}$ is spanned by the set of collision invariants:
$$
\mathcal{N}=\mbox{Span}\{\sqrt{\mu}, v\sqrt{\mu}, |v|^2\sqrt{\mu}\}\,,
$$
that is, $(\mathcal{L}g, g)_{L^2(\mathbb{R}^3_v)}=0$ if and only if $g \in \mathcal{N}$. We also let $\mathcal{N}^\perp$ denote the orthogonal space of $\mathcal{N}$ with respect to the standard inner product $(\cdot, \cdot)_{L^2(\mathbb{R}^3_v)}$.

Let us recall that the non-isotropic norm introduced in the series work of Alexandre-Morimoto-Ukai-Xu-Yang. Here for simplicity we use [AMUXY] to denote the references \cite{amuxy3}, \cite{amuxy4-2}, \cite{amuxy4-3}, \cite{amuxy7}.
\begin{equation}\label{triple norm}
\begin{aligned}
|\!|\!| g|\!|\!|^2=& \iiint_{\mathbb{R}^3_v\times \mathbb{R}^3_{v_*} \times \mathbb{S}^2} B(v-v_*, \sigma)\mu_*^2\, (g'-g)^2 \mathrm{d}\sigma\mathrm{d}v_*\mathrm{d}v \, \\
&+
\iiint_{\mathbb{R}^3_v\times \mathbb{R}^3_{v_*} \times \mathbb{S}^2} B(v-v_*, \sigma) g^2_* ({\mu'}\,\, - {\mu}\,\,)^2\, \mathrm{d}\sigma\mathrm{d}v_* \mathrm{d}v .
\end{aligned}
\end{equation}
See also Gressman-Strain \cite{G-S} for another equivalent definition of this norm.

Using this non-isotropic norm, we define that for $N \in \mathbb{N}$,
$$
\|f\|^2_{\cX^N(\RR^6_{x,v})}=\sum_{|\alpha|\le N}\int_{\RR^3_x}|\!|\!|\p^{\alpha}_xf|\!|\!|^2 \mathrm{d}x.
$$

Let us recall the macro-micro decomposition of solutions
$$
g = \pP g + (\iI-\pP)g = g_1 + g_2\,,
$$
where $\pP$ is the orthogonal projection to $\mathcal{N}$. $g_1=\pP g$ is called the macroscopic projection of $g(t, x, v)$,
\begin{equation}\label{Pg}
\pP g =\{a(t, x)+v\cdot b(t, x)+|v|^2c(t, x)\}\sqrt{\mu},\quad \cA (g)=(a,b,c)\,,
\end{equation}
while $g_2=(\iI-\pP)g$ is called the kinetic part of $g$.

Notice that
\begin{align*}
\|g\|_{H^N(\RR^3_x; L^2(\RR^3_v))}^2&\sim
\|\cA(g)\|_{H^N(\RR^3_x)}^2+\|g_2\|_{H^N(\RR^3_x; L^2(\RR^3_v))}^2,
\\
\|g\|^2_{\cX^{N}(\RR^6_{x,v})}&\sim \|\cA(g)\|_{H^N(\RR^3_x)}^2+\|g_2\|_{\cX^N(\RR^6)}^2.
\end{align*}

We introduce the following temporal energy functional and dissipation rate functional respectively
\begin{equation}\label{norms}
\begin{aligned}
&\cE^2_N(g)=\|g\|_{H^N(\RR^3_x; L^2(\RR^3_v))}^2=\|g_1\|_{H^N(\RR^3_x; L^2(\RR^3_v))}^2+
\|g_2\|_{H^N(\RR^3_x; L^2(\RR^3_v))}^2\\
&\qquad\qquad\sim
\|\cA(g)\|_{H^N(\RR^3_x)}^2+\|g_2\|_{H^N(\RR^3_x; L^2(\RR^3_v))}^2,
\\
&\cC_{N}(g)= \|\nabla_x\cA(g)\|_{H^{N-1}(\RR^3_x)}\,\sim\,\|\nabla_x g_1\|_{H^{N-1}(\RR^3_x; L^2(\RR^3_v))},
\\&
\cD_N(g)=\|g_2\|_{\cX^{N}(\RR^6_{x,v})}\,.
\end{aligned}
\end{equation}
Remark that
\begin{equation}\label{C-D}
\cC_N\le \cE_N.
\end{equation}
We also define the following weighted Sobolev spaces: let $\langle v \rangle = (1+ |v|^2)^{\frac{1}{2}}$,
\begin{equation}\nonumber
L^2_l(\mathbb{R}^3_v) = \{g\in \mathcal{S}'(\mathbb{R}^3_v): \|g\|_{L^2_l(\mathbb{R}^3_v)}= \|\langle v \rangle^l g\|_{L^2(\mathbb{R}^3_v)} < +\infty \}\,.
\end{equation}

\subsection{Main Theorems}
The main theorems are stated in the following. The first theorem is on the global existence of the Boltzmann equation uniform with respect to the Knudsen number $\eps$, and the second is on the incompressible Navier-Stokes limit as $\eps \rightarrow 0$ taken in the solutions $g_\eps$ of the Boltzmann equation \eqref{e2} which is constructed in the first theorem.

\begin{theorem}\label{theorem1}
Assume that the collision kernel $B(\cdot,\cdot)$ satisfies \eqref{kernel form}. It also satisfies for non-cutoff case \eqref{non-cutoff kernel} with $0<s<1, \gamma > \max \{-3, -\frac32-2s\}$, and \eqref{Grad} for cutoff case with $\gamma > -3$. Then for $N\ge 2$ and  $0 < \eps < 1$, there exists a $\delta_0>0$, independent of $\eps$, such that if $\|g_{\varepsilon, 0}\|_{H^N(\RR^3_x; L^2(\RR^3_v))}\le \delta_0$, the Cauchy problem \eqref{e2} admits a global solution
$$
g_\varepsilon\in L^\infty([0, +\infty); H^N(\RR^3_x; L^2(\RR^3_v)))
$$
with the global energy estimate:
\begin{equation}\label{global estimates}
  \sup_{t \geq 0} \cE^2_N(t) + c_0\int^\infty_0 \frac{1}{\eps^2}\cD^2_N(t)\,\mathrm{d}t + c_0\int^\infty_0 \cC^2_N(t)\,\mathrm{d}t \leq \cE^2_N(0)\,,
\end{equation}
here $c_0>0$ is independent of $\eps$.
\end{theorem}

The next theorem is about the limit to the incompressible Navier-Stokes-Fourier equation:
\begin{equation}\label{INS}
\begin{cases}
   \p_t \mathrm{u} + \mathrm{u}\!\cdot\! \grad \mathrm{u} + \grad p = \nu \Delta_{\!x}\mathrm{u}\,,\\
   \grad\!\cdot\! \mathrm{u}=0\,,\\
   \p_t \theta + \mathrm{u}\!\cdot\! \grad \theta = \kappa \Delta_{\!x}\theta\,,
\end{cases}
\end{equation}
where the viscosity and heat conductivity are given by
$$\nu=\tfrac{1}{15}(\sqrt{\mu}A_{ij}, \sqrt{\mu}\widehat{A}_{ij})_{L^2(\mathbb{R}^3_v)}\,,\quad \kappa =\tfrac{2}{15}(\sqrt{\mu}B_i, \sqrt{\mu}\widehat{B}_i)_{L^2(\mathbb{R}^3_v)}
$$respectively. Here $A_{ij}=v_i v_j - \frac{|v|^2}{3}, B_i = v_i(\frac{|v|^2}{2}-\frac{3}{2})$, and $\mathcal{L}\widehat{A}_{ij}= A_{ij}, \mathcal{L}\widehat{B}_i= B_i$.

\begin{theorem}\label{theorem2}
Let the collision kernel satisfy the same assumption as in Theorem \ref{theorem1}. Let $0<\eps<1$, $N \geq 2$ and $\delta_0 > 0$ be as in the Theorem \ref{theorem1}. For any $(\rho_0, u_0, \theta_0)\in H^N(\mathbb{R}^3_x)$ with $\|(\rho_0, u_0, \theta_0)\|_{H^N(\mathbb{R}^3_x)} < \frac{\delta_0}{2}$, and $\tilde{g}_{\eps,0} \in \mathcal{N}^\perp$ with $\|\tilde{g}_{\eps,0}\|_{H^N(\mathbb{R}^3_x;L^2(\mathbb{R}^3_v))}< \frac{\delta_0}{2}$, let
\begin{equation}\label{initial data}
  g_{\eps,0}(x,v) = \{\rho_0(x) + \mathrm{u}_0(x)\cdot v+ \theta_0(x)(\tfrac{|v|^2}{2}-\tfrac{3}{2})\}\sqrt{\mu} + \tilde{g}_{\eps,0}(x,v)\,.
\end{equation}
Let $g_\eps$ be the family of solutions to the Boltzmann equation \eqref{e2} constructed in Theorem \ref{theorem1}. Then,
\begin{equation}
   g_\eps \rightarrow \mathrm{u}\!\cdot\!v+ \theta(\tfrac{|v|^2}{2}-\tfrac{5}{2}) \quad \text{as}\quad\! \eps\rightarrow 0\,,
\end{equation}
where the convergence is weak$\mbox{-}\star$ for $t$, strongly in $H^{N-\eta}(\mathbb{R}^3_x)$ for any $\eta>0$, and weakly in $L^2(\mathbb{R}^3_v)$, and $(u,\theta) \in C([0,\infty);H^{N-1}(\mathbb{R}^3_x))\cap L^\infty([0,\infty);H^{N}(\mathbb{R}^3_x))$ is the solution of the incompressible Navier-Stokes-Fourier equation \eqref{INS} with initial data:
\begin{equation}\label{IC}
   \mathrm{u}|_{t=0}= \mathcal{P}\mathrm{u}_0(x)\,, \quad \theta|_{t=0}= \tfrac{3}{5}\theta_0(x)-\tfrac{2}{5}\rho_0(x)\,,
\end{equation}
where $\mathcal{P}$ is the Leray projection. Furthermore, the convergence of the moments holds: as $\eps\rightarrow 0$,
\begin{equation}
\begin{aligned}
   &\mathcal{P}(g_\eps, v\sqrt{\mu})_{L^2(\mathbb{R}^3_v)} \rightarrow \mathrm{u} \quad \text{in}\quad\! C([0,\infty);H^{N-1-\eta}(\mathbb{R}^3_x))\cap L^\infty([0,\infty);H^{N-\eta}(\mathbb{R}^3_x))\,,\\
   &(g_\eps, (\tfrac{|v|^2}{5}-1)\sqrt{\mu})_{L^2(\mathbb{R}^3_v)} \rightarrow \theta \quad \text{in}\quad\! C([0,\infty);H^{N-1-\eta}(\mathbb{R}^3_x))\cap L^\infty([0,\infty);H^{N-\eta}(\mathbb{R}^3_x))\,,
\end{aligned}
\end{equation}
 for any $\eta>0$.
\end{theorem}

\subsection{Historical Remarks}

The fluid limits from the Boltzmann equations have been gotten a lot of interest in the previous decades. The main contributions are the rigorous justifications of the incompressible Navier-Stokes and Euler equations. There are basically two directions based on the two different contexts of the solutions to the Boltzmann equations: the first, in the context of DiPerna-Lions renormalized solutions; the second, the classical solutions.

For the first direction, after DiPerna-Lions's renormalized solution to the Boltzmann equation with Grad's cutoff kernel \cite{D-L}, which are the only solutions known to exist globally without any restriction on the size of the initial data. See also the extension to the non-cutoff kernels by Alexandre-Villani \cite{al-3}. From late 80's, Bardos-Golse-Levermore initialized the program (BGL Program in brief) to justify Leray's solutions to the incompressible Navier-Stokes equations from DiPerna-Lions' renormalized solutions \cite{BGL1}, \cite{BGL2}. They proved the first convergence result with 5 additional technical assumptions. After 10 years effects by Bardos, Golse, Levermore, Lions and Saint-Raymond, see \cite{BGL3},\cite{LM3},\cite{LM4},\cite{GL}, the first complete convergence result without any additional compactness assumption was proved by Golse and Saint-Raymond in \cite{Go-Sai04} for cutoff Maxwell collision kernel, and in \cite{Go-Sai09} for hard cutoff potentials. Later on, it was extended by Levermore-Masmoudi \cite{LM} to include soft potentials. Recently Arsenio got the similar results for non-cutoff case \cite{Arsenio}.

The BGL program says that, given any $L^2\mbox{-}$bounded functions $(\rho_0, u_0, \theta_0)$, and for any physically bounded initial data (as required in DiPerna-Lions solutions) $F_{\eps,0}= \mu + \eps \sqrt{\mu}g_{\eps,0}$, such that suitable moments of the fluctuation $g_{\eps,0}$, say, $(\mathcal{P}(g_{\eps,0}, v\sqrt{\mu})_{L^2(\mathbb{R}^3_v)}, (g_{\eps,0}, (\tfrac{|v|^2}{5}-1)\sqrt{\mu})_{L^2(\mathbb{R}^3_v)})$ converges in the sense of distributions to $(\mathcal{P}u_0, \tfrac{3}{5}\theta_0-\tfrac{2}{5}\rho_0)$, the corresponding DiPerna-Lions solutions are $F_\eps(t,x,v)= \mu + \eps \sqrt{\mu}g_{\eps}(t,x,v)$. Then the fluctuations $g_{\eps}$ has weak compactness, such that the corresponding moments of $g_\eps$ converge weakly in $L^1$ to $(u,\theta)$ which is a Leray solution to the incompressible Navier-Stokes equation whose viscosity and heat conductivity coefficients are determined by microscopic information, with initial data $(\mathcal{P}u_0, \tfrac{3}{5}\theta_0-\tfrac{2}{5}\rho_0)$. Under some situations, for example the well-prepared initial data or in bounded domain with suitable boundary condition, the convergence could be strong $L^1$.

We emphasize that the BGL program indeed gave a new proof of Leray's solutions to the incompressible Navier-Stokes equation, in particular the energy inequality of which can be derived from the entropy inequality of the Boltzmann equation. Any a priori information of the Navier-Stokes equation is not needed, and completely derived from the microscopic Boltzmann equation. In this sense, BGL program is spiritually a part of Hilbert's 6th problem: derive and justify the macroscopic fluid equation from the microscopic kinetic equations.

The second direction on the fluid limits of Boltzmann equations is based on the Hilbert expansion  and in the context of classical solutions. It was started from Nishida and Caflisch's work on the compressible Euler limit \cite{Nishida}, \cite{Caflisch}, \cite{KMN}. After then this process was used in justifications for the incompressible limits, for examples, \cite{DEL-89} and \cite{GY-06}. In \cite{DEL-89}, De Masi-Esposito-Lebowitz considered  Navier-Stokes limit in dimension 2. More recently, using the nonlinear energy method, in \cite{GY-06} Y. Guo justified the Navier-Stokes limit (and beyond, i.e. higher order terms in Hilbert expansion). These results basically say that, given the initial data which is needed in the classical solutions of the Navier-Stokes equation, it can be constructed the solutions of the Boltzmann equation of the form $F_\eps = \mu + \eps \sqrt{\mu}(g_1+ \eps g_2 + \cdots + \eps^n g_\eps)$, where $g_1, g_2, \cdots $ can be determined by the Hilbert expansion, and $g_\eps$ is the error term. In particular, the first order fluctuation $g_1 = \rho_1 + u_1\!\cdot\! v + \theta_1(\frac{|v|^2}{2}-\frac{3}{2})$, where $(\rho_1, u_1, \theta_1)$ is the solutions to the incompressible Navier-Stokes equations.

Besides the mathematical techniques are quite different with the BGL program, (the BGL program uses entropy and weak compactness methods, and the second direction use energy method), philosophically, as mentioned above, BGL program does not assume any a priori information of the fluid equation, and derives the fluid equation from solutions of the Boltzmann equation, and consequently gives a solution to the fluid equation. The second direction go the opposite direction, say, they employ the solutions of the fluid equations, and construct the solutions of the Boltzmann equation {\em near} the infinitesimal Maxwellian $g_1$ which is determined by the solutions of the fluid equations (while the higher order term $g_i, i \geq 2$ are determined by linear fluid equations). In other words, if $g_1, g_2, \cdots $ are given by the solutions of the fluid equations, when the Knudsen number $\eps$ small enough, one can construct solutions of the Boltzmann equation of the form $F_\eps = \mu + \eps \sqrt{\mu}(g_1+ \eps g_2 + \cdots + \eps^n g_\eps)$.

The main purpose of present work is trying to study the fluid dynamic limits of the Boltzmann equation along the philosophy of the first direction in the context of classical solutions. The first work in this problem is Bardos-Ukai\cite{b-u}. They started from the scaled Boltzmann equation \eqref{e2} for cut-off hard potentials, and proved the global existence of classical solutions $g_\eps$ uniformly in $0< \eps <1$. The key feature of Bardos-Ukai's work is that they only need the smallness of the initial data, and did not assume the smallness of the Knudsen number $\eps$. After having the uniform in $\eps$ solutions $g_\eps$, taking limits can provide a classical solution of the incompressible Navier-Stokes equation with small initial data.

Bardos-Ukai's approach heavily depends on the sharp estimate especially the spectral analysis on the linearized Boltzmann operator $\cL$, and the semigroup method. Methodologically it is a linear method. They only treated the hard potentials case. It seems that it is hardly extended to soft potential cutoff, and even harder for the non-cutoff case, since it is well-known for those cases, the operator $\cL$ has continuous spectrum.

In the present paper, we consider much larger class of collision kernels for both cut-off and non-cutoff cases. We use the nonlinear energy method, in particular the nice properties of the non-isotropic norm defined in \eqref{triple norm}, which was recently developed in the series of works by [AMUXY]. We proved in Theorem \ref{theorem1} the uniform in $\eps$ global existence of the Boltzmann equation with or without cutoff assumption and established the global energy estimates. Then taking limit as $\eps\rightarrow 0$, proved the incompressible Navier-Stokes limit in Theorem \ref{theorem2}. Our result in fact give another proof of the small initial data classical solution to the incompressible Navier-Stokes equation. Furthermore, for the non-cutoff kernels, since the solutions established in [AMUXY] have full regularity, we expect that the solutions to the Navier-Stokes equation will have higher order regularities. This will be discussed in a future paper.

This paper is organized as follows: the next section is devoted to the local existence. In section 3, the uniform energy estimate and the global existence is established. In the final section, the incompressible Navier-Stokes limit is proved.

\section{Construction of Local Solutions}\label{local-existence}
\smallskip
\setcounter{equation}{0}

\subsection{Preparations}

For the convenience, we collect some know results about the collision operators. In the rest of the paper, we use the notation $a\lesssim b$ which means that there exists a generic constant (independent of $\eps$) $C$ such that $ a \leq C b $.
\begin{proposition}
The following estimates holds:
\begin{itemize}
\item For any $\gamma>-3, 0<s<1$, there exists two generic constants $C_1, C_2 >0$ such that
\begin{equation}\label{equivalent}
C_1 \big\{ \|g\|^2_{H^s_{\gamma/2}(\mathbb{R}^3_v)}+\|g\|^2_{L^2_{s+\gamma/2}(\mathbb{R}^3_v)} \big\}
\leq |\!|\!| g|\!|\!|^2 \leq C_2 \|g\|^2_{H^s_{s+\gamma/2}(\mathbb{R}^3_v)},
\end{equation}
and
\begin{equation}\label{coercive}
C_1|\!|\!|(\iI-\pP)g|\!|\!|^2 \leq (\mathcal{L} g, g)_{L^2(\RR^3_v)} \leq C_2 |\!|\!|g|\!|\!|^2.
\end{equation}

\item  For any $0<s<1$ and  $\gamma > \max \{-3, -3/2-2s\}$, there exists $C > 0$ such that
\begin{equation}\label{trilinear}
\Big |(\Gamma (f,g),h )_{L^2(\RR^3_v)}\Big| \leq C \|f\|_{L^2(\mathbb{R}^3_v)} |\!|\!|g|\!|\!|\,\, |\!|\!|h|\!|\!|,
\end{equation}
and
\begin{align}\label{trilinear-b}
\Big |(\Gamma (f,g),h )_{L^2(\RR^3_v)}\Big| & \leq C \Big\{\|f\|_{L^2_{s+\gamma/2}{(\mathbb{R}^3_v)}} |\!|\!|g|\!|\!|+\|g\|_{L^2_{s+\gamma/2}{(\mathbb{R}^3_v)}} |\!|\!|f|\!|\!|\\
&\quad+\min\{\|f\|_{L^2_{s+\gamma/2}{(\mathbb{R}^3_v)}} \|g\|_{L^2{(\mathbb{R}^3_v)}},\,\,
\|g\|_{L^2_{s+\gamma/2}{(\mathbb{R}^3_v)}} \|f\|_{L^2{(\mathbb{R}^3_v)}}\}\Big\}
|\!|\!|h|\!|\!|\, .\notag
\end{align}
\item For the cutoff case \eqref{Grad}, we have that \eqref{equivalent} holds true with $s=0$ and the trilinear upper bounded estimate \eqref{trilinear} holds true for $\gamma>-3$.
\end{itemize}
\end{proposition}
The estimates \eqref{equivalent}, \eqref{coercive} and \eqref{trilinear-b} were proved in \cite{amuxy4-2}, and \eqref{trilinear} was proved in \cite{amuxy7}. For the cutoff case, the trilinear upper bounded estimate is just Theorem 3 of \cite{guo-T}. {The rest of this manuscript will focus on the more difficult non-cutoff case.}

Next, we prepare some lemmas about the upper bounded estimate. The first is the following Galiardo-Nirenberg type inequality which was proved in \cite{amuxy4-2} (See Lemma 6.1 there.)

\begin{lemm}
Assume that $N \geq 3$ and let $\p^\alpha=\p^\alpha_x, \alpha \in \NN^3, |\alpha| \le N$. Then
\begin{equation}\label{2.1112}
\|\p^\alpha \cA^2\|_{L^2(\RR^3_x)}\lesssim \|\nabla_x \cA
\|_{H^{N-1}(\RR^3_x)}\|\cA\|_{H^{N}(\RR^3_x)},
\end{equation}
\end{lemm}

The next is an estimate on the nonlinear collision operator $\Gamma$ in terms of temporal energy functional and dissipation rate.

\begin{lemma}\label{triple-estimate}
Under the assumption of Theorem \ref{theorem1}, we have, for any $N\ge 2$,
\begin{equation}\label{upper-B}
\big(\Gamma(g, g),\,\, h\big)_{H^N(\RR^3_x; L^2(\RR^3_v))}\lesssim \cE_N(g) \{\cC_N(g)+\cD_N(g)\}\cD_N(h).
\end{equation}
\end{lemma}

\begin{proof}
In the following, we fix an index $|\alpha|\le N$, choose any indices $\alpha_1$ and $\alpha_2$ such that $\alpha_1 +\alpha_2 =\alpha$, and fix any
 $\varphi_k$, $\varphi_m$ in $\cN$.
Note that $(\p_x^{\alpha}\Gamma(g,\,g), \p_x^{\alpha}h_1)_{L^2(\RR^6)}=0$, we have
 \begin{equation}
(\p_x^{\alpha}\Gamma(g,\, g), \p_x^{\alpha} h)_{L^2(\RR^6)}
=(\p_x^{\alpha}\Gamma(g,\,g), \p_x^{\alpha}h_2)_{L^2(\RR^6)}
 =J^{11}+J^{12}+J^{21}+J^{22},
\end{equation}
where
\begin{equation}
J^{ij}=(\p_x^{\alpha}\Gamma(g_i,\, g_j), \p_x^{\alpha} h_2)_{L^2(\RR^6)}.
\end{equation}

\noindent
{\bf Estimation of $J^{11}$.}
We shall estimate, for $\varphi_k,\varphi_m\in\cN$,
\begin{align*}
&J^{11} \sim
\int_{\RR^3_x}(\p_x^{\alpha}\cA^2(g))
(\Gamma(\varphi_k,\varphi_m), \p_x^{\alpha} h_2)_{L^2(\RR^3_v)}dx.
\end{align*}
Then \eqref{trilinear} yielding
\begin{align*}
\Big|(\Gamma(\varphi_k,\varphi_m), \p_x^{\alpha} h_2)_{L^2(\RR^3_v)}\Big|
&
\lesssim |\!|\!|\p_x^{\alpha} h_2|\!|\!|,
\end{align*}
Now \eqref{2.1112} implies for $|\alpha|\le N$ and $N\ge 2$,
\begin{align*}
|J^{11}|&
\lesssim \|\p^\alpha \cA^2(g)\|_{L^2(\RR^3_x)}
\| h_2\|_{\cX^N(\RR^6)}
\\
&\lesssim
\|\cA(g)\|_{H^{N}(\RR^3_x)}\|\nabla_x\cA(g)\|_{H^{N-1}(\RR^3_x)}
\|h_2\|_{\cX^{N}(\RR^6)}\\
&\lesssim
\cE_N(g)\cC_N(g)\cD_N(h).
\end{align*}

\noindent
{\bf Estimation of $J^{12}$:}
Notice
\begin{align*}
 J^{12}&\sim \int_{\RR^3_x}(\p_x^{\alpha_1}\cA(g))
(\Gamma(\varphi_k, \p_x^{\alpha_2}g_2), \p_x^{\alpha} h_2)_{L^2(\RR^3_v)}dx.
\end{align*}
Again, \eqref{trilinear}, yielding
  \begin{align*}
| J^{12}|&
 \lesssim \int_{\RR^3_x}|\p_x^{\alpha_1}\cA(g)|\,\,
|\!|\!|\p_x^{\alpha_2} g_2|\!|\!|\,\,
 \ |\!|\!|\p_x^{\alpha} h_2|\!|\!|dx\, ,
\end{align*}
the Sobolev embedding theorem gives
\begin{align*}
| J^{12}|&\lesssim
\|\cA(g)\|_{H^N(\RR^3)}\|g_2\|_{\cX^{N}(\RR^6)}
\|h_2\|_{\cX^{N}(\RR^6)}
\lesssim \cE_N(g)\cD_N(g)\cD_N(h).
\end{align*}

\noindent
{\bf Estimation of $J^{21}$:}
\begin{align*}
 J^{21}&\sim \int_{\RR^3_x}(\p_x^{\alpha_2}\cA(g))
(\Gamma(\p_x^{\alpha_1}g_2, \varphi_k), \p_x^{\alpha} h_2)_{L^2(\RR^3_v)}dx.
\end{align*}
If $\alpha_2\neq0$, \eqref{trilinear} and Sobolev inequality yields
\begin{align*}
| J^{21}|&\lesssim
\|\nabla_x\cA(g)\|_{H^{N-1}(\RR^3)}
\|g_2\|_{H^N(\RR^3_x; L^2(\RR^3_v))}
\|h_2\|_{\cX^{N}(\RR^6)}\\
&\lesssim \cE_N(g)\cC_N(g)\cD_N(h).
\end{align*}
If $\alpha_2=0$, \eqref{equivalent} and \eqref{trilinear-b} yields
\begin{align*}
\big|(\Gamma(\p_x^{\alpha}g_2, \varphi_k), \p_x^{\alpha} h_2)_{L^2(\RR^3_v)}\big|
&\lesssim \|\p_x^{\alpha}g_2\|_{L^2_{s+\gamma/2}(\RR^3_v)}\,\,
|\!|\!|\p_x^{\alpha}h_2|\!|\!|\\
&
\lesssim |\!|\!|\p_x^{\alpha}g_2|\!|\!|\,\,
|\!|\!|\p_x^{\alpha}h_2|\!|\!|,
\end{align*}
thus
$$
| J^{21}|\lesssim
\|\cA(g)\|_{H^2(\RR^3)}
\|g_2\|_{\cX^{N}(\RR^6)}\|h_2\|_{\cX^{N}(\RR^6)}\lesssim \cE_N(g)\cD_N(g)\cD_N(h).
$$

\noindent
{\bf Estimation of $J^{22}$:}
\begin{align*}
 J^{22}&\sim \int_{\RR^3_x}
(\Gamma(\p_x^{\alpha_1}g_2, \p_x^{\alpha_1}g_2), \p_x^{\alpha} h_2)_{L^2(\RR^3_v)}dx,
\end{align*}
\eqref{trilinear-b} yields
\begin{align*}
| J^{22}|&\lesssim
\|g_2\|_{H^N(\RR^3_x; L^2(\RR^3_v))}
\|g_2\|_{\cX^{N}(\RR^6)}\|h_2\|_{\cX^{N}(\RR^6)}\lesssim \cE_N(g)\cD_N(g)\cD_N(h).
\end{align*}
Now, combining the above estimates yields the estimate \eqref{part2-nonlinear5} and this completes the proof of the Lemma \ref{triple-estimate}.
\end{proof}

\subsection{Linear Problem}
We consider the following linear Cauchy problem
\begin{equation}\label{linear-a}
\begin{cases}
\partial_tg+\frac 1 \varepsilon v\!\cdot\!\nabla_{x}g+\frac 1{\varepsilon^2}\mathcal{L}g=\frac{1}{\varepsilon}
\Gamma(f, f),\\
g|_{t=0}=g_0,
\end{cases}
\end{equation}
where $f$ is a given function. We study the existence of solution in the function space $H^N(\RR^3_x; L^2(\RR^3_v))$.

\begin{proposition}\label{existence-linear}
Under the assumption of Theorem \ref{theorem1},
let~$g_0\in H^N(\RR^3_x; L^2(\RR^3_v))$ with $N\ge 2$ and for some $T>0$, $f$ satisfies
$$
\sup_{0\le t\le T}\cE^2_N(f(t))+\int^T_0\cD^2_N(f(t))\,\mathrm{d}t<+\infty.
$$
Then the Cauchy problem \eqref{linear-a} admits an unique solution
$$
g\in L^\infty ([0, T]; H^N(\RR^3_x; L^2(\RR^3_v))).
$$
\end{proposition}

\begin{proof} We prove the existence of solution to the Cauchy problem \eqref{linear-a} by the Hahn-Banach theorem. We rewrite \eqref{linear-a} into the following form
\begin{equation}\label{4.1.1}
{\mathcal T} g \equiv \partial_t g + \frac{1}{\varepsilon}v\!\cdot\!\nabla_x g + \frac{1}{\varepsilon^2}\cL g= \frac 1 \varepsilon\Gamma(f, f) , \ g(0) =g_0.
\end{equation}
For $h\in C^\infty ([0,T];\, {\mathcal S}(\RR^6_{x,v}) )$ with $h(T)
=0$, we define ${\mathcal T}^*_{N}$ through
$$
\Big(g,\,\, {\mathcal T}^*_{N} \, h\Big)_{L^2([0,\,T];
H^N(\RR^3_x; L^2(\RR^3_v)))} = \Big({\mathcal T}\,g,\, h\Big)_{L^2([0,\,T];
H^N(\RR^3_x; L^2(\RR^3_v)))}\,\, ,
$$
so that ${\mathcal T}^*_{N} $ is the adjoint of the operator
${\mathcal T}$ in the Hilbert space $L^2 ([0,\,T]; H^N(\RR^3_x; L^2(\RR^3_v)))$.

Set
$$
{\mathbb W} = \left\{ w = {\mathcal T}^*_{N}\, h ;\,\,h\in
C^\infty ([0,T];\, {\mathcal S}(\RR^6_{x,v}) )\,\,\mbox{with}\,\,
h(T) =0 \right\},
$$
which  is a dense subspace of $L^2 ([0,\,T];
\,H^N(\RR^3_x; L^2(\RR^3_v)))$. And we also have
$$
{\mathcal T}^\ast_{N} (h) = -\partial_t  h -
\frac{1}{\varepsilon}(v\!\cdot\!\nabla_x ) h + \frac{1}{\varepsilon^2}\cL\, h.
$$
Then
\begin{align*}
\Big(h,\,\, {\mathcal T}^*_{N} \, h\Big)_{H^N(\RR^3_x; L^2(\RR^3_v))}
=&\frac 12 \frac{\mathrm{d}}{\mathrm{d}t}|| h (t) ||^2_{H^N(\RR^3_x; L^2(\RR^3_v))} +\frac{1}{\varepsilon}
\Big(v\!\cdot\!\nabla_x\, h ,\, h\Big)_{
H^N(\RR^3_x; L^2(\RR^3_v))}\\
& +\frac{1}{\varepsilon^2} \Big(\cL (h) , h\Big)_{H^N(\RR^3_x; L^2(\RR^3_v)))}.
\end{align*}
Note that the second term above vanishes and the estimate \eqref{coercive}, we have
\begin{align*}
\int^T_t \Big|\Big(h,\,\, {\mathcal T}^*_{N} \,
h\Big)_{H^N(\RR^3_x; L^2(\RR^3_v))}\Big|\,\mathrm{d}t &\geq \frac12|| h(t)||^2_{H^N(\RR^3_x; L^2(\RR^3_v))} + \frac{c}{\varepsilon^2}\int^T_t \cD^2_N(h(s))\,\mathrm{d}s\, .
\end{align*}
Thus, for all $0<t<T$,
\begin{align*}
& || h(t)||^2_{H^N(\RR^3_x; L^2(\RR^3_v))}+ \frac{c}{\varepsilon^2}\int^T_t\cD^2_N(h(s))\,\mathrm{d}s\\
&\leq  || {\mathcal T}^*_{N} (h) ||_{L^2([t,\,T];
H^N(\RR^3_x; L^2(\RR^3_v)))} || h||_{L^2([t,\,T]; H^N(\RR^3_x; L^2(\RR^3_v)))}\,.
\end{align*}
Hence, we get
\begin{equation}\label{4.1.2+0}
\| h\|_{L^\infty([0,\,T]; H^N(\RR^3_x; L^2(\RR^3_v)))}  \leq C \sqrt{T} ||  {\mathcal T}^*_{N}
(h) ||_{L^2([0,\,T]; H^N(\RR^3_x; L^2(\RR^3_v)))}\,,
\end{equation}
and
\begin{equation}\label{4.1.2}
\frac{1}{\varepsilon}\big(\int^T_0\cD^2_N(h(s))\,\mathrm{d}s\big)^{1/2}\leq C ||  {\mathcal T}^*_{N}
(h) ||_{L^2([0,\,T]; H^N(\RR^3_x; L^2(\RR^3_v)))}\,.
\end{equation}

Next, we define a functional ${\mathcal G}$ on ${\mathbb W}$ as follows
$$
{\mathcal G}(w) = \frac1 \varepsilon(\Gamma(f, f),\,h)_{L^2([0,\,T]; H^N(\RR^3_x; L^2(\RR^3_v)))} + ( g_0
, h(0))_{H^N(\RR^3_x; L^2(\RR^3_v))}.
$$

Then, using \eqref{C-D} and \eqref{upper-B}
\begin{align*}
 |{\mathcal G}(w)| &\leq \frac1\varepsilon \int^T_0\{\cE^2_N(f)+\cE_N(f)\cD_N(f)\}\cD_N(h)\,\mathrm{d}t\\
&\qquad\qquad+\|g_0\|_{H^N(\RR^3_x; L^2(\RR^3_v))}\|h(0)\|_{H^N(\RR^3_x; L^2(\RR^3_v))}\\
&\leq \frac1\varepsilon\sup_{0<t<T}\cE_N(f)\{ \sup_{0<t<T}\cE_N(f)+\big(\int^T_0\cD^2_N(f)\,\mathrm{d}t\big)^{1/2}\}\big(\int^T_0
\cD^2_N(h)\,\mathrm{d}t\big)^{1/2}\\
&\qquad\qquad+\|g_0\|_{H^N(\RR^3_x; L^2(\RR^3_v))}\|h\|_{L^\infty([0, T]; H^N(\RR^3_x; L^2(\RR^3_v)))},
\end{align*}
finally, \eqref{4.1.2+0} and \eqref{4.1.2} imply
\begin{align*}
 |{\mathcal G}(w)| \leq C(f, g_0) || {\mathcal T}^*_{N} (h) ||_{L^2 ([0,\,T];\,
H^N(\RR^3_x; L^2(\RR^3_v)))}\leq C|| w||_{L^2 ([0,\,T];\, H^N(\RR^3_x; L^2(\RR^3_v)))}\, ,
\end{align*}
where
$$
C(f, g_0)=\sup_{0<t<T}\cE_N(f)\{ \sup_{0<t<T}\cE_N(f)+\big(\int^T_0\cD^2_N(f)\,\mathrm{d}t\big)^{1/2}\}
+ \sqrt{T}\|g_0\|_{H^N(\RR^3_x; L^2(\RR^3_v))}.
$$

Thus,
${\mathcal G}$ is a continuous linear functional on $\Big({\mathbb
W};\,\|\,\cdot\,\|_{L^2 ([0,\,T];\, H^N(\RR^3_x; L^2(\RR^3_v)))}\Big)$. So by the Hahn-Banach Theorem, $\mathcal{G}$ can be extended from $\mathbb{W}$ to $L^2 ([0,\,T];\, H^N(\RR^3_x; L^2(\RR^3_v)))$. From the Riesz representation theorem, there exists $g\in L^2 ([0,\,T];\, H^N(\RR^3_x; L^2(\RR^3_v)))$ such that for any
$w\in{\mathbb W}$,
$$
{\mathcal G}(w)= \big(g,\,w\big)_{L^2 ([0,\,T];\, H^N(\RR^3_x; L^2(\RR^3_v)))}\,.
$$
 For any $h\in C^\infty ([0,T]; {\mathcal S}
(\RR^6_{x,v}) )$ with $h(T) =0$, we have
\begin{align*}
\Big(g,\,\, {\mathcal T}^*_{N} \, h\Big)_{L^2([0,\,T];
H^N(\RR^3_x; L^2(\RR^3_v)))} &= \frac1\varepsilon\big(\Gamma(f, f),\, h\big)_{L^2([0,\,T]; H^N(\RR^3_x; L^2(\RR^3_v)))}\\
&\qquad\qquad
+ \big(g_0 ,\, h(0)\big)_{H^N(\RR^3_x; L^2(\RR^3_v))},
\end{align*}
and by the definition of the operator ${\mathcal T}^*_{N}$, we have also
\begin{equation}\label{weak-S}
\Big({\mathcal T}\,g,\,\, \tilde h\Big)_{L^2([0,\,T]; L^2(\RR^6_{x,v}))} =
\frac1\varepsilon\big(\Gamma(f, f),\, \tilde h\big)_{L^2([0,\,T]; L^2(\RR^6_{x,v}))} + \big(g_0 ,\,
\tilde h(0)\big)_{L^2(\RR^6_{x,v})},
\end{equation}
where
$$
\tilde h  = \Lambda^{2N}_x h \in C^\infty ([0, T];\,
{\mathcal S}(\RR^6_{x,v}) )\,\,\, \mbox{with}\,\,\, \tilde h (T )=0,
$$
where $\Lambda=(1-\Delta_{x})^{\frac 12}$.
Since $\Lambda^{2N} $ is an isomorphism on
$\big\{h:h\in C^\infty ([0, T];\, {\mathcal S}(\RR^6_{x,v}) )$
with $ h (T )=0\big\}$, then $g\in L^2 ([0,\,T];\,\,
H^N(\RR^3_x; L^2(\RR^3_v)))$ is a solution of the Cauchy problem (\ref{4.1.1}). Using \eqref{weak-S}, we can also prove
\begin{equation}\label{bbb}
\sup_{0<t<T}\cE_N^2(g(t))+\frac 1{\varepsilon^2}\int^T_0 \cD^2_N(g(t))\,\mathrm{d}t\lesssim \tilde{C}(f, g_0),
\end{equation}
where $\tilde{C}(f, g_0)$ is similar to $C(f, g_0)$ and independents on $0<\varepsilon\le 1$. Indeed,
$$
\tilde{C}(f, g_0) = \cE_N^2(g_0) + \{ \sup_{0<t<T}\cE_N^2(f(t))+ \int^T_0 \cD^2_N(f(t))\,\mathrm{d}t\}^2\,.
$$
\end{proof}

\subsection{Local existence for nonlinear problem}
We consider now the following iteration
\begin{equation}\label{linear}
\begin{cases}
\partial_tg^{n+1}+\frac 1 \varepsilon v\cdot\nabla_{x}g^{n+1}+\frac 1{\varepsilon^2}\mathcal{L}g^{n+1}=\frac{1}{\varepsilon}
\Gamma(g^{n}, g^{n}),\\
g^{n+1}|_{t=0}=g_0,
\end{cases}
\end{equation}
with $g^0\equiv0$.

\begin{proposition}\label{iteration-solution}
There exists $0<\delta_0\le 1, 0<T\le 1$, such that for any $0<\varepsilon\le1, g_0\in H^N(\RR^3_x; L^2(\RR^3_v)), N\ge 2$ with
\begin{equation}\label{initial-d}
\|g_0\|_{H^N(\RR^3_x; L^2(\RR^3_v))}\le \delta_0,
\end{equation}
then the iteration problem \eqref{linear} admits a sequence of solution $\{g^n\}_{n\ge 1}$ satisfy
\begin{equation}\label{iteration-n}
\sup_{t\in [0, T]}\cE^2_N(g^n)
+\frac{1}{\varepsilon^2}\int^T_0\cD^2_N(g^n)\,\mathrm{d}t\le 4 \delta^2_0.
\end{equation}
\end{proposition}

\begin{proof}Notice that for the linear Cauchy problem
\eqref{linear}, for given $g^n$ satisfy \eqref{iteration-n},
the existence of $g^{n+1}$ is assured by the Proposition \ref{existence-linear}. So that it is enough to prove \eqref{iteration-n} by induction, using \eqref{coercive} and \eqref{upper-B}, there exists $C>0$ such that
\begin{align*}
\frac{\mathrm{d}}{\mathrm{d}t}\cE^2_N(g^{n+1})+\frac{1}{\varepsilon^2}
\cD^2_N(g^{n+1})
&\le\frac{C}{\varepsilon}\big\{\cE^2_N(g^n) +\cE_N(g^n)\cD_N(g^{n})\big\}\cD_N(g^{n+1})\\
&\le \frac{1}{2\varepsilon^2}\cD^2_N(g^{n+1})+2C^2
\cE^2_N(g^n)\big\{\cE^2_N(g^n) +\cD^2_N(g^{n})\big\}.
\end{align*}
Thus, we get
\begin{align*}
&\frac{\mathrm{d}}{\mathrm{d}t}\cE^2_N(g^{n+1})
+\frac{1}{\varepsilon^2}\cD^2_N(g^{n+1})\\
&\qquad\le 4C^2\cE^2_N(g^n)\big\{\cE^2_N(g^n)+
\cD^2_N(g^n)\big\}.
\end{align*}
Integration on $[0, T]$ with $T\le 1$,
\begin{align*}
&\sup_{t\in [0, T]}\cE^2_N(g^{n+1})
+\frac{1}{\varepsilon^2}\int^T_0\cD^2_N(g^{n+1})\,\mathrm{d}t\le \cE^2_N(g_0)\\
&\qquad+4C^2\sup_{t\in [0, T]}\cE^2_N(g^{n})\big\{\sup_{t\in [0, T]}\cE^2_N(g^{n})+
\int^T_0\cD^2_N(g^n)\big\},
\end{align*}
we complete the proof of the Proposition if we chose $\delta_0$ such that
$$
1+64 C^2\delta^2_0\le 4.
$$
\end{proof}

Finally, from the uniform estimate \eqref{iteration-n}, we can prove the convergence of $\{g^n\}$, thus the following local existence results through a standard argument as in \cite{amuxy3}.

\begin{theorem}\label{local-solution}
There exists $\delta_0>0, T>0$, such that for any $0<\varepsilon<1, g_{\varepsilon,0}\in H^N(\RR^3_x; L^2(\RR^3_v))$, $N\ge 2$  with
\begin{equation}\label{initial-d}
\|g_{\varepsilon, 0}\|_{H^N(\RR^3_x; L^2(\RR^3_v))}\le \delta_0,
\end{equation}
then the Cauchy problem \eqref{e2} admits an unique solution $g_\varepsilon \in L^\infty ([0, T]; H^N(\RR^3_x; L^2(\RR^3_v)))$ satisfy
\begin{equation}\label{local-est}
\sup_{t\in [0, T]}\cE^2_N(g_\varepsilon)
+\frac{1}{\varepsilon^2}\int^T_0\cD^2_N(g_\varepsilon)\,\mathrm{d}t\le 4 \delta^2_0.
\end{equation}
\end{theorem}
\begin{proof}

It is enough to prove that $\{g^n\}$ is a Cauchy sequence in $L^\infty ([0, T]; L^2(\RR^6_{x,v}))$. Set $w^n=g^{n+1}-g^n$
and deduce from \eqref{linear},
$$
\begin{cases}
\partial_tw^{n}+\frac 1 \varepsilon v\cdot\nabla_{x}w^{n}+\frac 1{\varepsilon^2}\mathcal{L}w^{n}=\frac{1}{\varepsilon}\Big[
\Gamma(g^{n}, g^{n})-\Gamma(g^{n-1}, g^{n-1})\Big],\\
w^{n}|_{t=0}=0.
\end{cases}
$$
Since, for any $h\in L^2$
\begin{align*}
&\Big(\Gamma(g^{n}, g^{n})-\Gamma(g^{n-1}, g^{n-1}), \mathbf{P}h\Big)_{L^2(\RR^3_v)}=0\,, \\
&\Gamma(g^{n}, g^{n})-\Gamma(g^{n-1}, g^{n-1})=\Gamma(g^{n}, w^{n-1})+
\Gamma(w^{n-1}, g^{n-1}),
\end{align*}
then
\begin{align*}
&\Big(\Gamma(g^{n}, g^{n})-\Gamma(g^{n-1}, g^{n-1}), w^n\Big)_{L^2(\RR^6)}\\
&=\Big(\Gamma(g^{n}, w^{n-1})+
\Gamma(w^{n-1}, g^{n-1}), w^n_2\Big)_{L^2(\RR^6)},
\end{align*}
using \eqref{trilinear} and  $H^N(\mathbb{R}^3_x)\hookrightarrow L^\infty(\mathbb{R}^3_x)$ for $N \geq 2$,
\begin{align*}
&\Big|\Big(\frac{1}{\varepsilon}(\Gamma(g^{n}, w^{n-1})+
\Gamma(w^{n-1}, g^{n-1})), w^n_2\Big)_{L^2(\RR^6)}\Big|\\
&\le \frac{C}{\varepsilon}\mathcal{E}_N(g^{n})(\|w^{n-1}_1\|_{\mathcal{X}^0}+\mathcal{D}_0(w^{n-1})) \mathcal{D}_0(w^{n})\\
&+\frac{C}{\varepsilon}
\mathcal{E}_0(w^{n-1})(\|g^{n-1}_1\|_{\mathcal{X}^N}+\mathcal{D}_N(g^{n-1})) \mathcal{D}_0(w^{n})\\
&\le C_\delta\mathcal{E}^2_N(g^{n})(\|w^{n-1}_1\|^2_{\mathcal{X}^0}+\mathcal{D}^2_0(w^{n-1})) +\frac{\delta}{\varepsilon^2}\mathcal{D}^2_0(w^{n})\\
&+C_\delta
\mathcal{E}^2_0(w^{n-1})(\|g^{n-1}_1\|^2_{\mathcal{X}^N}+\mathcal{D}^2_N(g^{n-1})).
\end{align*}
Thus, fix a small $\delta>0$, we get
\begin{align*}
\frac{\mathrm{d}}{\mathrm{d}t}\|w^{n}\|^2_{L^2(\RR^6)}+\frac{1}{\varepsilon^2}
\cD^2_0(w^{n})
&\le C_\delta\mathcal{E}^2_N(g^{n})(\|w^{n-1}_1\|^2_{\mathcal{X}^0}+\mathcal{D}^2_0(w^{n-1})) \\
&+C_\delta
\mathcal{E}^2_0(w^{n-1})(\|g^{n-1}_1\|^2_{\mathcal{X}^N}+\mathcal{D}^2_N(g^{n-1})).
\end{align*}
Note that
$$
\|w^{n-1}_1\|^2_{\mathcal{X}^0}\le C \mathcal{E}^2_0(w^{n-1}),\quad
\|g^{n-1}_1\|^2_{\mathcal{X}^N}\le C \mathcal{E}^2_N(g^{n-1}),
$$
we have proved
\begin{align*}
&\|w^{n}\|^2_{L^\infty([0, T]; L^2(\RR^6))}+\frac{1}{\varepsilon^2}
\int^T_0\cD^2_0(w^{n})\,\mathrm{d}t\\
&\le C \sup_{t\in [0, T]}\mathcal{E}^2_N(g^{n})\Big(T\sup_{t\in [0, T]}\mathcal{E}^2_0(w^{n-1})+\int^T_0\mathcal{D}^2_0(w^{n-1})\,\mathrm{d}t\Big) \\
&+C \sup_{t\in [0, T]}
\mathcal{E}^2_0(w^{n-1})\Big(T\sup_{t\in [0, T]}\mathcal{E}^2_N(g^{n-1})+\int^T_0\mathcal{D}^2_N(g^{n-1})\,\mathrm{d}t\Big).
\end{align*}
Using now \eqref{iteration-n} with $\delta_0>0$ small enough, we get that for any $0<\varepsilon\le 1$
\begin{align*}
\sup_{t\in [0, T]}\mathcal{E}^2_0(w^{n})+\frac{1}{\varepsilon^2}
\int^T_0\cD^2_0(w^{n})\,\mathrm{d}t
&\le \frac 12 \Big(\sup_{t\in [0, T]}\mathcal{E}^2_0(w^{n-1})+\frac{1}{\varepsilon^2}\int^T_0
\mathcal{D}^2_0(w^{n-1})\,\mathrm{d}t\Big) .
\end{align*}
Thus we have proved that  $\{g^n\}$ is a Cauchy sequence in $L^\infty ([0, T]; L^2(\RR^6_{x,v}))$.

Combining with the estimate \eqref{iteration-n} and interpolation,  $\{g^n\}$ is a Cauchy sequence in $L^\infty([0,T],H^{N-\eta}(\mathbb{R}^3_x,L^2(\mathbb{R}^3_v))$ for any $\eta > 0$ and the limit is in $L^\infty([0,T],H^{N}(\mathbb{R}^3_x,L^2(\mathbb{R}^3_v))$. Finally the estimate \eqref{local-est} follows from weak lower semicontinuity.
\end{proof}

\section{Uniform estimate and global  solutions}\label{section-existence}
\setcounter{equation}{0}

Let $g_\eps$ be a local solution of Cauchy problem \eqref{e2},
we use the continuation argument of local solutions to
prove the existence of global solutions in the space $H^N(\RR^3_x; L^2(\RR^3_v)), {  N\ge 2}$. For this, we need to carry  out the smallness assumption of initial data.

\subsection{Microscopic Energy Estimate}
We study firstly the estimate on the microscopic component $g_2$ in the function space ${H^N(\RR^3_x; L^2(\RR^3_v))}$.
For notational simplification, we drop the sub-index $\varepsilon$ of $g$, and also drop $g$ in the notations
$\mathcal{A}, \cE_N, \cC_N, \cD_N$. Actually, we shall establish

\begin{proposition}\label{microenergy}
Let $g\in L^\infty([0,T];H^N(\RR^3_x; L^2(\RR^3_v)))$ be a solution of the equation \eqref{e2} constructed in Theorem \ref{local-solution}, then there exists a constant $C$ independent of $\eps$ such that the following estimate holds:
\begin{equation}\label{micro}
\frac{\mathrm{d}}{\mathrm{d}t}\cE^2_N+\frac{1}{\varepsilon^2}\cD^2_N
\le C\Big\{\frac{1}{\varepsilon}\cE_N \cD_N^2+(\cE_N\cC_N)^2 \Big\}.
\end{equation}
\end{proposition}
\begin{proof}
We apply $\p_x^\alpha $ to \eqref{e2} and take the $L^2(\RR^6_{x,v})$ inner product with
$\p_x^\alpha g$. Since the inner product including $v\!\cdot\!\nabla_xg$ vanishes by
integration by parts, we get
\begin{align}\label{part2-microinner}
\frac 12\frac{\mathrm{d}}{\mathrm{d}t} \cE^2_N +\frac{1}{\varepsilon^2}\sum_{|\alpha|\le N}(\mathcal{L}\p_x^\alpha g,
 \p_x^\alpha g)_{L^2(\RR^6_{x,v})}=\frac{1}{\varepsilon}\sum_{|\alpha|\le N}
(\p_x^\alpha \Gamma(g, g), \p_x^{\alpha}g)_{L^2(\RR^6_{x,v})}.
\end{align}
Note that the above identity makes sense is guaranteed by  \eqref{local-est}. In view of \eqref{coercive},  we have,
$$
\sum_{|\alpha|\le N}(\mathcal{L}\p_x^\alpha g,
 \p_x^\alpha g)_{L^2(\RR^6_{x,v})}\ge C_1 \|g_2\|^2_{\cX^N(\RR^6_{x,v})}=C_1 \cD^2_N.
 $$
Lemma \ref{triple-estimate} implies that for $|\alpha|\le N$,
\begin{equation}\label{part2-nonlinear5}
\begin{aligned}
 &\frac{1}{\eps}\Big|(\p_x^\alpha \Gamma(g, g), \p_x^{\alpha}g)_{L^2(\RR^6_{x,v})}
\Big| \\
\leq &\frac{C_2}{\eps} \cE_N(\cC_N\cD_N+\cD_N^2)\leq
\frac{C_2}{\eps}\cE_N\cD_N^2+  \frac{1}{4 \eta}(\cE_N\cC_N)^2  +\eta\frac{1}{\varepsilon^2}\cD_N^2.
\end{aligned}
\end{equation}
Taking $\eta = \frac{C_1}{2}$, then Proposition \ref{microenergy} can be concluded by plugging  these two estimates into \eqref{part2-microinner}.

\end{proof}

\subsection{Macroscopic energy estimates}

We study now the energy estimate for the macroscopic part $\mathbf{P}g$ where $g$ is a solution of the equation \eqref{e2}. First we decompose the equation \eqref{e2} into microscopic and macroscopic parts, i.e. rewrite it into the following equation
\begin{align}\label{macro-e}
\partial_t\{a+bv+c|v|^2\}\mu^{1/2}&+\frac 1 \varepsilon v\cdot\nabla_{x}\{a+bv+c|v|^2\}\mu^{1/2}\\
&=-\partial_tg_2-\frac 1 \varepsilon v\cdot\nabla_{x}g_2-\frac 1{\varepsilon^2}\mathcal{L}g_2+\frac{1}{\varepsilon}
\Gamma(g_\varepsilon, g_\varepsilon),\notag
\end{align}

\begin{lemm}\label{part2-abc2}
Let $\p^\alpha=\p^\alpha_x, \alpha \in \NN^3, |\alpha| \le N$. If $g$ is a solution of the Boltzmann equation \eqref{e2}, and $\mathcal{A}=(a,b,c)$ defined in \eqref{Pg}, then
\begin{equation}\label{2.111}
\varepsilon\|\p_t\p^\alpha \cA\|_{L^2(\RR^3_x)}\lesssim  \mathcal{C}_N+
\mathcal{D}_N\,.
\end{equation}
\end{lemm}

\begin{proof}
Let $g$ be a solution to a solution of the scaled Boltzmann equation \eqref{e2}. The local conservation laws are given by
\begin{equation}\label{conservation laws}
\begin{aligned}
&\partial_t a = \frac{1}{2\eps}(v\!\cdot\! \grad g_2, |v|^2\sqrt{\mu})_{L^2(\mathbb{R}^3_v)}\,,\\
&\partial_t b + \frac{1}{\eps}(\grad a + 5 \grad c)= -\frac{1}{\eps}(v\!\cdot\! \grad g_2, v\sqrt{\mu})_{L^2(\mathbb{R}^3_v)}\,,\\
&\partial_t c + \frac{1}{3\eps}\grad\!\cdot\! b= -\frac{1}{\eps}(v\!\cdot\! \grad g_2, |v|^2\sqrt{\mu})_{L^2(\mathbb{R}^3_v)}\,,
\end{aligned}
\end{equation}
from which we can deduce that
$$
\|\p_t\p^\alpha \cA\|_{L^2(\RR^3_x)}
 \lesssim \frac 1 \varepsilon   \|\nabla_x\p^\alpha  \cA\|_{L^2(\RR^3_x)}+
\frac 1 \varepsilon \|\nabla_x\p^\alpha g_2\|_{L^2(\RR^3_{x}; L^2_{s+\gamma/2}(\RR^3_v))}.
$$
This completes the proof of the lemma by using \eqref{equivalent}. In particular, from the first equation of \eqref{conservation laws}, we have
\begin{equation}\label{local a}
\varepsilon\|\p_t\p^\alpha a\|_{L^2(\RR^3_x)}\lesssim
\mathcal{D}_N\,.
\end{equation}
\end{proof}

Next, we put the so-called 13-moments
\begin{equation}\label{ecomp}
  \{ e_j \}^{13}_{j=1}=
\left\lbrace  \mu^{1/2} , v_i \mu^{1/2} ,v_iv_j\mu^{1/2} , v_i|v|^2\mu^{1/2} \right\rbrace\,.
\end{equation}
This set of functions spans a 13-dimensional subspace of $L^2(\RR^3_v)$. Let
 $\{e ^*_k\}_{k=1}^{13}$ be  a corresponding bi-orthogonal basis,
i.e.  a basis such that
\[
(e^*_j, e^*_k)_{L^2(\RR^3_v)}=\delta_{j,k}, \qquad j,k=1,\cdots,13,
\]
hold.
Of course $e^*_k$ is given as a linear combination of
\eqref{ecomp}.
It is well-known \cite{guo-1} that the macroscopic component
$
g_1=\pP g\sim \cA=(a,b,c),
$
satisfies the following set of equations
\begin{equation}\label{part2-macroeq}
\left\{
\begin{array}{rrl}
v \,| v|^2 \mu^{1/2} :&\frac{1}{\varepsilon}\nabla_x c &= -\partial_t r_c+\frac{1}{\varepsilon}m_c+\frac{1}{\varepsilon^2}l_c + \frac{1}{\varepsilon}h_c,
\\
v^2_i \mu^{1/2}:&\partial_t c +\frac{1}{\varepsilon}\partial_ib_i &= -\partial_t r_i+\frac{1}{\varepsilon}m_i+\frac{1}{\varepsilon^2}l_i + \frac{1}{\varepsilon}h_i ,
\\
 v_iv_j \mu^{1/2}:&\frac{1}{\varepsilon}\partial_ib_j + \frac{1}{\varepsilon} \partial_j b_i &
= -\partial_t r_{ij}+\frac{1}{\varepsilon}m_{ij}+\frac{1}{\varepsilon^2}l_{ij} + \frac{1}{\varepsilon}h_{ij} , \quad i\neq j,
\\
 v_i \mu^{1/2} :&\partial_t b_i + \frac{1}{\varepsilon}\partial_i a &
= -\partial_t r_{bi}+\frac{1}{\varepsilon}m_{bi}+\frac{1}{\varepsilon^2}l_{bi} + \frac{1}{\varepsilon}h_{bi},
\\
 \mu^{1/2} :& \partial_t a &= -\partial_t r_a+\frac{1}{\varepsilon}m_a+\frac{1}{\varepsilon^2}l_a +\frac{1}{\varepsilon} h_a .
\end{array}
\right.
\end{equation}
In fact,
one obtains each equation of the second column, if one multiplies \eqref{macro-e} by such an $e^*_j$ and integrating in
$v$, where
 $r_c, \cdots, h_a$ are the inner products  of the form
\begin{align}\label{rmh}
&r=(g_2,e^*)_{L^2(\RR^3_v)},\,\,
m=-(v\cdot\nabla_x g_2, e^*)_{L^2(\RR^3_v)},\, h=(\Gamma(g,g),e^*)_{L^2(\RR^3_v)},\\
&l=-(\mathcal{L} g_2,e^*)_{L^2(\RR^3_v)},\label{l}
\end{align}
in which  $e^*$ stands for the corresponding  $e^*_j$.
Here and below we drop the index $c, i, ij,b_i,a$ since the computations are similar.

The next lemma gives estimates on the various terms involved in the right-hand side of the macroscopic system \eqref{part2-macroeq}. Moreover, in the left-hand side of the following estimates, $r, m, l$ and $h$ stand for one of the corresponding terms of the macroscopic system as explained above.

\begin{lemma} \label{part2-rmh}
Let $r,m,l,h$ be the ones  defined by \eqref{rmh} with $e^*$ replaced by
any linear combination of the basis functions $e^*_j$.
 Let $\p^\alpha=\p^\alpha_x$, $\p_i=\p_{x_i}$, $|\alpha|\le N-1$. Then, one has
\begin{align}
&\|\p_i\p^\alpha r\,\|_{L^2(\RR^3_x)}
\lesssim\min\{\|g_2\|_{H^{N}(\RR^3_x; L^2(\RR^3_v))},\,\mathcal{D}_{N}\},\label{part2-D1}
\\
&\|\p^\alpha m\,\|_{L^2(\RR^3_x)}
\lesssim\min\{\|g_2\|_{H^{N}(\RR^3_x; L^2(\RR^3_v))},\,\mathcal{D}_{N}\},\label{part2-D2}
\\
&\|\p^\alpha l\,\|_{L^2(\RR^3_x)}
\lesssim \min\{\|g_2\|_{H^{N-1}(\RR^3_x; L^2(\RR^3_v))},\,\mathcal{D}_{N-1}\}
\label{part2-D3},
\\
 &\label{part2-Dh}
\|\p^\alpha   h\|_{L^2(\RR^3_x)}\lesssim \cE_{N-1}(\cC_{N-1}+ \cD_{N-1}).
\end{align}
\end{lemma}
\noindent
\begin{proof}
Since $e^*_j$ can be expressed as a linear combination
of basis functions $\{e_i\}$, we may compute $r, m, h$ with $e^*$ any linear combination $e$ of $\{e_i\}$. Remark that $e=\tilde{e}\sqrt{\mu}$, then, for any $\ell\in\RR$,
\begin{align*}
\|\p_i\p^\alpha r\,\|_{L^2(\RR^3_x)}
&=
\| (\p_i\p^\alpha g_2, e)_{L^2(\RR^3_v)}\|_{L^2(\RR^3_x)}
\\&\lesssim
\|\ \|\p_i\p^\alpha  g_2\|_{L^2_\ell(\RR^3_v)}
\|_{L^2(\RR^3_{x})}\lesssim \|g_2\|_{H^{N}(\RR^3_x; L^2_\ell(\RR^3_v))},
\end{align*}
\begin{align*}
\|\p^\alpha l\,\|_{L^2(\RR^3_x)}
&=
\| (\cL(\p^\alpha g_2), e)_{L^2(\RR^3_v)}\|_{L^2(\RR^3_x)}
\\&=
\| ( \p^\alpha g_2,\cL^* e)_{L^2(\RR^3_v)}\|_{L^2(\RR^3_x)}
\lesssim \| \p^\alpha g_2\|_{L^{2}(\RR^3_x; L^2_\ell(\RR^3_v))},
\end{align*}
\begin{align*}
\|\p^\alpha   m\,\|_{L^2_{x}}&=
 \| (\nabla_x\p^\alpha g_2, \,v\,e)_{L^2(\RR^3_v)}\|_{L^2(\RR^3_x)}\\
& \lesssim \| \nabla_x\p^\alpha g_2 \|_{L^2(\RR^3_x; L^2_\ell(\RR^3_v))}
\lesssim \|g_2\|_{H^{N}(\RR^3_x; L^2_\ell(\RR^3_v))},
\end{align*}
chose $\ell=0$ and $\ell=s+\gamma/2$, thus \eqref{equivalent} imply \eqref{part2-D1},
\eqref{part2-D2} and \eqref{part2-D3}.
The estimate \eqref{part2-Dh} is a direct consequence of \eqref{upper-B}, since $h$ is computed as follows.
\[
h=(\Gamma(g, g), e).
\]
\end{proof}

Using the previous three lemmas, we are now able to prove our first differential inequality, which estimates $\alpha +1$ derivatives of the macroscopic part $\cA$ in terms of the microscopic part $g_2$, for $|\alpha | \leq N-1$. Below, on the left-hand side, $r$ stands for the vector of all the previous $r$.

\begin{lemma}\label{Macro-estimate}
Let $|\alpha|\le N-1$, and let $g$ be a solution of the scaled Boltzmann equation \eqref{e2}. Then there exists a positive constant $\widetilde{C}$ independent of $\eps$, such that the following estimate holds:
\begin{align}\label{part2-pabc}
\varepsilon\frac{\mathrm{d}}{\mathrm{d}t}\Big\{(\p^\alpha r,&\nabla_x \p^\alpha (a, - b, c))_{L^2(\RR^3_x)}
+(\p^\alpha b, \nabla_x \p^\alpha a)_{L^2(\RR^3_x)}\Big\}
+ C^2_{N}
\\
&\hspace{2cm}\notag
\le \widetilde{C}\Big\{\frac{1}{\varepsilon^2} \cD^2_{N}+\cE_{N}(\cC^2_{N}+\cD^2_{N})\Big\}.
\end{align}
\end{lemma}

\begin{proof}
Recall that
$$
\cC^2_{N}=\|\nabla_x \p^\alpha \cA\|_{L^2(\RR^3_x)}^2=\|\nabla_x \p^\alpha a\|_{L^2(\RR^3_x)}^2+
\|\nabla_x \p^\alpha b\|_{L^2(\RR^3_x)}^2+\|\nabla_x \p^\alpha c\|_{L^2(\RR^3_x)}^2.
$$
{\bf (a) Estimate of $\nabla_x\p^\alpha  a$}. From the macroscopic equations \eqref{part2-macroeq},
\begin{align*}
& \|\nabla_x \p^\alpha a\|_{L^2(\RR^3_x)}^2
=(\nabla_x \p^\alpha a,\nabla_x \p^\alpha a)_{L^2(\RR^3_x)}
\\
&
= (\p^\alpha (-\varepsilon\p_t b-\varepsilon\p_tr + m+\frac 1 {\varepsilon}l+h),\nabla_x \p^\alpha a)_{L^2(\RR^3_x)}
\\&\lesssim \varepsilon R_1+|(\p^\alpha m, \nabla_x \p^\alpha a)_{L^2(\RR^3_x)}|+\frac{1}{\varepsilon}|(\p^\alpha l, \nabla_x \p^\alpha a)_{L^2(\RR^3_x)}|+|(\p^\alpha h, \nabla_x \p^\alpha a)_{L^2(\RR^3_x)} |.
\end{align*}
Here,
\begin{align*}
\varepsilon R_1&=-\varepsilon (\p^\alpha \p_tb+\p^\alpha \p_tr,\nabla_x \p^\alpha a)_{L^2(\RR^3_x)}
\\&=-\varepsilon \frac{\mathrm{d}}{\mathrm{d}t}(\p^\alpha (b+r),\nabla_x \p^\alpha a)_{L^2(\RR^3_x)}-
\varepsilon(\nabla_x\p^\alpha (b+r), \p_t \p^\alpha a)_{L^2(\RR^3_x)}\,.
\end{align*}
Note that the estimate \eqref{local a}, $\varepsilon(\nabla_x\p^\alpha r, \p_t \p^\alpha a)_{L^2(\RR^3_x)} \lesssim \cD^2_N$, and
\begin{equation}\nonumber
\varepsilon(\nabla_x\p^\alpha b, \p_t \p^\alpha a)_{L^2(\RR^3_x)} \lesssim \eta \|\grad \p^\alpha b\|^2_{L^2(\mathbb{R}^3_x)} + \frac{1}{4\eta}\cD^2_N\,.
\end{equation}
Furthermore, Lemma \ref{part2-rmh} implies that
\begin{align*}
&|(\p^\alpha m, \nabla_x \p^\alpha a)_{L^2(\RR^3_x)}|\le \cD_{N}\|\nabla_x\cA\|_{H^{N-1}(\RR^3_x)}\lesssim \cD_{N}\cC_{N},\\
&\frac{1}{\eps}|(\p^\alpha l, \nabla_x \p^\alpha a)_{L^2(\RR^3_x)}|\lesssim \frac{1}{\eps}\cD_{N-1}\cC_{N} \leq \frac{1}{\eps}\cD_{N}\cC_{N},\\
&|(\p^\alpha h, \nabla_x \p^\alpha a)_{L^2(\RR^3_x)}|\lesssim \cE_{N-1}(\cC_{N-1}+\cD_{N-1})\cC_{N}\,.
\end{align*}
Hence, for some small $0<\eta<1$,
\begin{align*}
&\varepsilon \frac{\mathrm{d}}{\mathrm{d}t}(\p^\alpha (b+r),\nabla_x \p^\alpha a)_{L^2(\RR^3_x)}+ \|\nabla_x \p^\alpha a\|_{L^2(\RR^3_x)}^2\notag\\
&\quad\lesssim
\eta \|\nabla_x\p^\alpha b\|^2_{L^2(\RR^3_x)} + \frac{1}{4\eta} \cD^2_{N}
+\eta (\cC^2_{N}
+ \cD^2_{N}\big)\\
&\quad\quad+\cD_{N}\mathcal{C}_{N}+
\frac 1 \varepsilon \cD_{N}\mathcal{C}_{N} +\cE_{N}(\cC_{N}+\cD_{N})\mathcal{C}_{N}\notag
\\
&\quad\lesssim
\frac{1}{4\eta \varepsilon^2} \cD^2_{N}
+\eta \cC^2_{N}+\cE_{N}(\cC^2_{N}+\cD^2_{N})\notag
\end{align*}
Thus,
\begin{align}
&\varepsilon \frac{\mathrm{d}}{\mathrm{d}t}(\p^\alpha (b+r),\nabla_x \p^\alpha a)_{L^2(\RR^3_x)}+ \|\nabla_x \p^\alpha a\|_{L^2(\RR^3_x)}^2\label{3.15}\\
&\quad\le \eta \cC^2_{N}+
 \frac{C}{\eta}\big\{\frac 1 {\varepsilon^2} \cD^2_{N}
+\cE_{N}(\cC^2_{N}+\cD^2_{N})\big\}\,,\notag
\end{align}
where $C>0$ independent of $\eps$ and $\eta$.

\noindent{\bf (b) Estimate of $\nabla_x\p^\alpha  b$}.
Recall $b=(b_1,b_2,b_3)$. From \eqref{part2-macroeq} ,
\begin{align*}
\Delta_x\p^\alpha b_i+\p^2_{i}\p^\alpha   b_i&=\p^{\alpha}\Big[
\sum_{j \ne i} \p_j ( \p_j b_i + \p_i b_j)
+ \p_i(2  \p_i b_i - \sum_{j\ne i} \p_j  b_j)\Big]
\\&=\p^{\alpha}\Big[
\nabla_x  (-\varepsilon\p_t r +m+\frac{1}{\varepsilon}l+h)
- \p_i(2\p_t c-2 \p_t c)\Big],
\end{align*}
where $r,l,h$ stands for  linear combinations of $r_i,l_i,h_i$ and $r_{ij},l_{ij},h_{ij}$ for $ i,j=1,2,3$ respectively. Then
\begin{align*}
&\|\nabla_x\p^\alpha  b_i\|_{L^2(\RR^3_x)}^2+\|\p_i\p^\alpha  b_i\|_{L^2(\RR^3_x)}^2=
-(\Delta_x\p^\alpha b_i+\p^2_{i}\p^\alpha   b_i, \p^\alpha  b_i)_{L^2(\RR^3_x)}\\
&=\varepsilon R_2+R_3+R_4+R_5,
\end{align*}
where
\begin{align*}
\varepsilon R_2&=-\varepsilon(\nabla_x\p^\alpha \p_t  r,\p^\alpha  b_i)_{L^2(\RR^3_x)}\\
&=-\varepsilon\frac{\mathrm{d}}{\mathrm{d}t}(\p^\alpha   r, -\nabla_x\p^\alpha  b_i)_{L^2(\RR^3_x)}
-\varepsilon(\p^\alpha   r,\p_t\nabla_x\p^\alpha  b_i)_{L^2(\RR^3_x)}
\\&\hspace*{1cm}\lesssim
-\varepsilon\frac{\mathrm{d}}{\mathrm{d}t}(\p^\alpha   r, \ \nabla_x\p^\alpha  b_i)_{L^2(\RR^3_x)}+
\frac{1}{4\eta} \cD^2_{N}
+\eta\varepsilon^2
\|\p_t\p^\alpha  b_i\|^2_{L^2(\RR^3_x)},
\\
R_3&=-(\p^\alpha   m, \ \nabla_x\p^\alpha  b_i)_{L^2(\RR^3_x)}\lesssim \frac{1}{4\eta} \cD^2_{N}
+
\eta
\|\nabla_x\p^\alpha  b_i\|^2_{L^2(\RR^3_x)}
,
\\
R_4&=-\frac{1}{\varepsilon}(\p^\alpha   l, \ \nabla_x\p^\alpha  b_i)_{L^2(\RR^3_x)}\lesssim
\frac{1}{\varepsilon}\cC_{N}\cD_{N} \leq \frac{1}{\varepsilon}\cC_{N}\cD_{N},
\\R_5&=-(\p^\alpha   h,\ \nabla_x\p^\alpha  b_i)_{L^2(\RR^3_x)}\lesssim
\cE_{N}(\cC^2_{N}+\cD^2_{N}).
\end{align*}
Thus
\begin{align}
&\varepsilon \frac{\mathrm{d}}{\mathrm{d}t}(\p^\alpha r, -\nabla_x \p^\alpha b)_{L^2(\RR^3_x)}+ \|\nabla_x \p^\alpha b\|_{L^2(\RR^3_x)}^2\label{3.16}\\
&\quad\le \eta \cC^2_{N}+
 \frac{C}{\eta}\big\{\frac 1 {\varepsilon^2} \cD^2_{N}
+\cE_{N}(\cC^2_{N}+\cD^2_{N})\big\}\,,\notag
\end{align}
where $C>0$ independent of $\eps$ and $\eta$.

\noindent{\bf (c) Estimate of $\nabla_x\p^\alpha  c$}. {}From \eqref{part2-macroeq} ,
 \begin{align*}
 \|\nabla_x \p^\alpha &c\|_{L^2(\RR^3_x)}^2
=(\nabla_x \p^\alpha c,\nabla_x \p^\alpha c)_{L^2(\RR^3_x)}\\
&= (\p^\alpha (-\varepsilon\p_tr+m+\frac{1}{\varepsilon}l+h),\nabla_x \p^\alpha c)_{L^2(\RR^3_x)}
\\&
\lesssim \varepsilon R_6+\eta \cC^2_{N}
+\cD^2_{N}+\frac{1}{\varepsilon^2}\cD^2_{N}+\cE_{N}(\cC^2_{N}+\cD^2_{N}),
\end{align*}
where
\begin{align*}
\varepsilon R_6&=-\varepsilon(\p^\alpha \p_tr,\nabla_x \p^\alpha c)_{L^2(\RR^3_x)}=
- \varepsilon\frac{\mathrm{d}}{\mathrm{d}t}(\p^\alpha r,\nabla_x \p^\alpha c)_{L^2(\RR^3_x)}-\varepsilon
(\nabla_x\p^\alpha r, \p_t \p^\alpha c)_{L^2(\RR^3_x)}
\\&\quad
\lesssim
-\varepsilon\frac{\mathrm{d}}{\mathrm{d}t}(\p^\alpha r,\nabla_x \p^\alpha c)_{L^2(\RR^3_x)}+
\frac{1}{4\eta}\, \cD^2_{N}
+\eta \varepsilon^2
\|\p_t\p^\alpha  c\|^2_{L^2(\RR^3_x)}.
\end{align*}
Thus
\begin{align}
&\varepsilon \frac{\mathrm{d}}{\mathrm{d}t}(\p^\alpha r, -\nabla_x \p^\alpha c)_{L^2(\RR^3_x)}+ \|\nabla_x \p^\alpha c\|_{L^2(\RR^3_x)}^2\label{3.17}\\
&\quad\le \eta \cC^2_{N}+
 \frac{C}{\eta}\big\{\frac 1 {\varepsilon^2} \cD^2_{N}
+\cE_{N}(\cC^2_{N}+\cD^2_{N})\big\}\,,\notag
\end{align}
where $C>0$ independent of $\eps$ and $\eta$.

By combining the above estimates \eqref{3.15}, \eqref{3.16}, \eqref{3.17} and taking $\eta>0$ sufficiently small, then we get the estimate \eqref{part2-pabc} uniformly for $0<\varepsilon<1$, thus complete the proof of Lemma \ref{Macro-estimate}.
\end{proof}

Based on the microscopic estimate \eqref{micro} and the macroscopic estimate \eqref{part2-pabc}, we can derive the uniform energy estimate. Take the following  form of linear combination of \eqref{micro} and \eqref{part2-pabc},
\begin{align*}
&\frac{\mathrm{d}}{\mathrm{d}t}\Big\{\cE^2_{N}+d_1 \varepsilon\sum_{|\alpha|\le N-1}
\Big((\p^\alpha r,\nabla_x \p^\alpha (a, - b, c))_{L^2(\RR^3_x)}
+(\p^\alpha b, \nabla_x \p^\alpha a)_{L^2(\RR^3_x)}\Big)\Big\}
+\frac{1}{\varepsilon^2}\cD_{N}^2\\+ &d_1\,\cC^2_{N}
\le C\big(\frac{1}{\varepsilon}\cE_{N} \cD^2_{N}+(\cE_{N}\cC_{N})^2\big)
+\frac{d_1 \widetilde{C}}{\varepsilon^2}\cD^2_{N}
+d_1 \widetilde{C}\cE_{N}(\cC^2_{N}+\cD^2_{N}).
\end{align*}

\bigskip

Choose $d_1$ firstly such that $1-d_1 \widetilde{C}>0$, we have for $0<\varepsilon<1$
\begin{align*}
&\frac{\mathrm{d}}{\mathrm{d}t}\Big[\cE^2_{N}+d_1 \varepsilon\sum_{|\alpha|\le N-1}
\Big((\p^\alpha r,\nabla_x \p^\alpha (a, - b, c))_{L^2(\RR^3_x)}
-(\p^\alpha b, \nabla_x \p^\alpha a)_{L^2(\RR^3_x)}\Big)\Big]
\\
+&(1-d_1 \widetilde{C})\frac{1}{\varepsilon^2}\cD^2_{N}+d_1\,\cC^2_{N}
\leq( C+d_1 \widetilde{C})\frac{1}{\varepsilon}\cE_{N}\cD^2_{N}+(C\cE_{N}+d_1 \widetilde{C})\cE_{N}\cC^2_{N} .
\end{align*}
Set
\begin{align*}
E^2_N=\Big[\cE^2_N+d_1 \varepsilon\sum_{|\alpha|\le N-1}
\Big((\nabla_x\p^\alpha r, \p^\alpha (a, - b, c))_{L^2(\RR^3_x)}
-(\p^\alpha b, \nabla_x \p^\alpha a)_{L^2(\RR^3_x)}\Big)\Big].
\end{align*}
\eqref{part2-D1} implies that
\begin{align*}
|\sum_{|\alpha|\le N-1}(\nabla_x\p^\alpha r, &\  \p^\alpha (a, - b, c))_{L^2(\RR^3_x)}
+(\p^\alpha b, \nabla_x \p^\alpha a)_{L^2(\RR^3_x)}|
\\&
\lesssim
\|g_2\|_{H^{N}(\RR^3_x; L^2(\RR^3_v))}^2
+\|\cA\|_{H^{N}(\RR^3_x)}^2= \cE^2_N,
\end{align*}
then we can choose $d_1>0$ small such that, for any $0<\varepsilon<1$
\begin{equation}\label{c1c2}
c_1 \cE_N\le E_N\le c_2\cE_N
\end{equation}
for some positive constants $c_1$ and $c_2$. Thus we prove the following theorem:
\begin{theo} {\rm \bf (Global Energy Estimate) \ }
\label{part2-energyineq}
For $N\geq 2$, if $g$ is a solution of the scaled Boltzmann equation \eqref{e2}, then there exists a constant $c_0>0$ independent of $\eps$ such that   if $E_N\le 1$, then
\begin{equation}\label{global energy}
\frac{\mathrm{d}}{\mathrm{d}t}E^2_N+\frac{1}{\varepsilon^2}\cD^2_N+\cC^2_N\le c_0 E_N\big\{\frac{1}{\varepsilon}\cD^2_N+
\cC^2_N\big\}
\end{equation}
holds as far as $g$ exists.
\end{theo}

\subsection{Proof of Theorem $\ref{theorem1}$} Now, we are ready to prove Theorem \ref{theorem1} by the usual continuation arguments.
\begin{proof}
 We choose the initial data $g_{0,\eps}$ such that
$$\cE_N(0)=\|g_{0,\eps}\|_{H^N(\mathbb{R}^3_x,L^2(\mathbb{R}^3_v))} \leq M\,,$$
where $M$ is defined as
\begin{equation}\label{small intial data}
 M = \min\{\delta_0, \tfrac{1}{c_2}, \tfrac{1}{4c_0 c_2}\}\,.
\end{equation}
Recall that $c_0, c_1, c_2$ are constants in \eqref{global energy} and \eqref{c1c2}, and $\delta_0$ appears in Theorem \ref{local-solution}. Note that $\cE_N(0) \leq M \leq \delta_0$, then from Theorem \ref{local-solution} there exists a solution $g\in L^\infty([0,T];H^N(\mathbb{R}^3_x),L^2(\mathbb{R}^3_v))$ for some $T>0$, and from the local estimate \eqref{local-est}, we have $\cE_N(t) \leq 2 M$ for $0 < t < T$. We define
$$T^* = \sup \{t \in \mathbb{R}^+ | \cE_N(t) \leq 2 M\}> 0 \,.$$
Note that on $[0, T]$ for $0< T < T^*$, $E_N(t) \leq c_2 \cE_N(t) \leq 2 c_2 M <1$. Then the global energy estimate \eqref{global energy} implies that
\begin{equation}
 \frac{\mathrm{d}}{\mathrm{d}t}E^2_N+(1-2c_0 M)\{\frac{1}{\varepsilon^2}\cD^2_N+\cC^2_N\} \leq 0\,.
\end{equation}
From the choice of $M$, $1- 2 c_0 M > \frac{1}{2}$. Thus
$$
E^2_N(T) + \frac{1}{2}\int^T_0 \{\frac{1}{\varepsilon^2}\cD^2_N+\cC^2_N\}\,\mathrm{d}t \leq E^2_N(0)\,,
$$
which implies $\cE_N(T) \leq \frac{c_2}{c_1} M$. Modify the initial data gives that $T^* = \infty$, Thus we finish the proof of Theorem \ref{theorem1}.
\end{proof}

\section{Limit to Incompressible Navier-Stokes-Fourier Equations}

\subsection{The limit from the global energy estimate}

Based on Theorem\ref{theorem1}, there exists a $\delta_0 > 0$, such that the Boltzmann equation \eqref{e2} admits a global solution $g_\eps$ with initial data
$$
  g_{\eps,0}(x,v) = \{\rho_0(x) + \mathrm{u}_0(x)\cdot v+ \theta_0(x)(\tfrac{|v|^2}{2}-\tfrac{3}{2})\}\sqrt{\mu} + \tilde{g}_{\eps,0}(x,v)\,,
$$
where $$\|(\rho_0\,, \mathrm{u}_0\,, \theta_0)\|_{H^N(\mathbb{R}^3_x)} \leq \delta_0\,,$$ and $$\tilde{g}_{\eps,o}\in \cN^\perp\,,\quad \text{with}\quad\! \|\tilde{g}_{\eps,o}\|_{H^N(\mathbb{R}^3_x,L^2(\mathbb{R}^3_v))} \leq \delta_0\,.$$
Furthermore, the global energy estimate \eqref{global estimates} holds, i.e.
\begin{equation}\label{global bound E}
  \sup_{t \geq 0}\cE^2_N(t) = \sup_{t \geq 0} \sum_{|\alpha| \leq N}\int_{\mathbb{R}^6_{x,v}}|\p^\alpha_x g_\eps(t)|^2\,\mathrm{d}v\mathrm{d}x \leq C\,,
\end{equation}
and
\begin{equation}\label{global bound D}
 \int^\infty_0 \cD^2_N(t)\,\mathrm{d}t= \sum_{|\alpha| \leq N}\int^\infty_0 \int_{\mathbb{R}^3_x}|\!|\!|\p^\alpha_x\{\mathbf{I}-\mathbf{P}\} g_\eps(t)|\!|\!|^2\,\mathrm{d}x\mathrm{d}t \leq C \eps^2\,,
\end{equation}
and
\begin{equation}\label{global bound C}
\int^\infty_0 \cC^2_N(t)\,\mathrm{d}t= \sum_{|\alpha| \leq N}\int^\infty_0  \int_{\mathbb{R}^6_{x,v}}|\p^\alpha_x \mathbf{P} g_\eps(t)|^2\,\mathrm{d}x\mathrm{d}v\mathrm{d}t \leq C\,.
\end{equation}

From the energy bound \eqref{global bound E}, there exists a $g_0 \in L^\infty([0, +\infty); H^N(\RR^3_x; L^2(\RR^3_v)))$, such that
\begin{equation}\label{weak convergence g-eps}
   g_\eps \rightarrow g_0\quad  \mbox{as}\quad\!\eps\rightarrow 0\,,
\end{equation}
where the convergence is weak$\mbox{-}\star$ for $t$, strongly in $H^{N-\eta}(\mathbb{R}^3_x)$ for any $\eta>0$, and weakly in $L^2(\mathbb{R}^3_v)$.

From the energy dissipation bound \eqref{global bound D} and the inequality \eqref{equivalent}, we have
\begin{equation}\label{g2 vanish}
  \{\mathbf{I}-\mathbf{P}\} g_\eps \rightarrow 0\,,\quad \text{in}\quad\! L^2([0, +\infty); H^N(\RR^3_x; L^2(\RR^3_v)))\quad\!\text{as}\quad\! \eps\rightarrow 0\,.
\end{equation}
Combining the convergence \eqref{weak convergence g-eps} and \eqref{g2 vanish}, we have $\{\mathbf{I}-\mathbf{P}\} g_0 =0\,.$ Thus, there exists $(\rho, \mathrm{u}, \theta)\in L^\infty([0, +\infty); H^N(\RR^3_x)$, such that
\begin{equation}\label{g0}
g_0(t,x,v) = \rho(t,x) + \mathrm{u}(t,x)\cdot v + \theta(t,x)(\tfrac{|v|^2}{2}-\tfrac{3}{2})\,.
\end{equation}

\subsection{The limiting equations}
Now we define the fluid variables as follows:
\begin{equation}\nonumber
\rho_\eps = (g_\eps, \sqrt{\mu})_{L^2_v}\,, \quad \mathrm{u}_\eps = (g_\eps, v\sqrt{\mu})_{L^2_v}\,, \quad \theta_\eps = (g_\eps, (\tfrac{|v|^2}{3}-1)\sqrt{\mu})_{L^2_v}\,,
\end{equation}
where $L^2_v$ denotes $L^2(\mathbb{R}^3_v)$. It follows from \eqref{weak convergence g-eps} that
\begin{equation}\label{weak convergence fluid}
   (\rho_\eps, \mathrm{u}_\eps, \theta_\eps) \rightarrow (\rho, \mathrm{u}, \theta)\quad \text{as}\quad\! \eps\rightarrow 0\,,
\end{equation}
where the convergence is weak$\mbox{-}\star$ for $t$, strongly in $H^{N-\eta}(\mathbb{R}^3_x)$ for any $\eta>0$, and weakly in $L^2(\mathbb{R}^3_v)$.

Taking inner products with the Boltzmann equation \eqref{e2} in $L^2_v$ by $\sqrt{\mu}, v\sqrt{\mu}$ and $(\frac{|v|^2}{3}-1)\sqrt{\mu}$ respectively gives the local conservation laws:
\begin{equation}\label{local cons-law}
\begin{cases}
   \p_t \rho_\eps + \tfrac{1}{\eps} \grad\!\cdot\! \mathrm{u}_\eps = 0\,, \\
   \p_t \mathrm{u}_\eps + \tfrac{1}{\eps} \grad(\rho_\eps + \theta_\eps) + \grad\!\cdot\! (\sqrt{\mu}\widehat{A}\,, \tfrac{1}{\eps}\mathcal{L} g_\eps)_{L^2(\mathbb{R}^3_v)} = 0\,,\\
   \p_t \theta_\eps + \tfrac{2}{3}\tfrac{1}{\eps} \grad\!\cdot\! \mathrm{u}_\eps + \tfrac{2}{3}\grad\!\cdot\! (\sqrt{\mu}\widehat{B}\,, \tfrac{1}{\eps}\mathcal{L} g_\eps)_{L^2(\mathbb{R}^3_v)} =0 \,.
\end{cases}
\end{equation}

\noindent{\bf Incompressibility and Boussinesq relation:} From the first equation of \eqref{local cons-law} and the global energy bound \eqref{global bound E}, it is easy to deduce
\begin{equation}\label{div u vanish}
  \grad\!\cdot\! \mathrm{u}_\eps \rightarrow 0 \quad \text{in the sense of distributions}\quad\!\text{as}\quad\! \eps\rightarrow 0\,.
\end{equation}
Combining with the convergence \eqref{weak convergence fluid}, we have
\begin{equation}\label{incompressibility}
   \grad\!\cdot\! \mathrm{u} = 0 \,.
\end{equation}
From the second equation of \eqref{local cons-law},
$$
   \grad(\rho_\eps + \theta_\eps) = -\eps \p_t \mathrm{u}_\eps + \grad\!\cdot\! (\mathcal{L}(\sqrt{\mu}\widehat{A}), \{\mathbf{I}-\mathbf{P}\}g_\eps)_{L^2(\mathbb{R}^3_v)}\,.
$$
From the global energy dissipation \eqref{global bound D}, it follows that
\begin{equation}\label{rho-theta vanish}
  \grad(\rho_\eps + \theta_\eps) \rightarrow 0 \quad \text{in the sense of distributions}\quad\!\text{as}\quad\! \eps\rightarrow 0\,,
\end{equation}
which gives the Boussinesq relation
\begin{equation}\label{Boussinesq}
  \grad(\rho + \theta) = 0\,.
\end{equation}

\noindent{\bf Convergence of $\tfrac{3}{5}\theta_\eps - \tfrac{2}{5}\rho_\eps$:}
The third equation minus $\frac{2}{3}$ times the first equation in \eqref{local cons-law} gives
\begin{equation}\label{theta equation}
 \p_t (\tfrac{3}{5}\theta_\eps - \tfrac{2}{5}\rho_\eps)  + \tfrac{2}{5}\grad\!\cdot\! (\sqrt{\mu}\widehat{B}\,, \tfrac{1}{\eps}\mathcal{L} g_\eps)_{L^2(\mathbb{R}^3_v)} =0\,.
\end{equation}
From the global energy estimate \eqref{global bound E}, we have that for almost every $t\in [0,\infty)$, $\|(\tfrac{3}{5}\theta_\eps - \tfrac{2}{5}\rho_\eps)(t)\|_{H^N(\mathbb{R}^3_x)} \leq C$. Then there exists a $\tilde{\theta} \in L^\infty([0,\infty;H^N(\mathbb{R}^3_x)))$, so that
\begin{equation}
   (\tfrac{3}{5}\theta_\eps - \tfrac{2}{5}\rho_\eps)(t) \rightarrow \tilde{\theta}(t)\quad \text{in}\quad\! H^{N-\eta}(\mathbb{R}^3_x)\,,
\end{equation}
for any $\eta > 0$ as $\eps\rightarrow 0$. Furthermore, using the equation \eqref{theta equation}, we can show the equi-continuity in $t$. Indeed, $[t_1, t_2]\subset [0,\infty)$, any test function $\chi(x)$ and $|\alpha| \leq N-1$,
\begin{equation}\label{equi-cont}
\begin{aligned}
   &\int_{\mathbb{R}^3_x} \left[\p^\alpha_x (\tfrac{3}{5}\theta_\eps - \tfrac{2}{5}\rho_\eps)(t_2)- \p^\alpha_x (\theta_\eps - \tfrac{2}{3}\rho_\eps)(t_1)\right]\chi(x) \,\mathrm{d}x\\
   = & \int^{t_2}_{t_1}\int_{\mathbb{R}^3_x} (\sqrt{\mu}\widehat{B}\,, \tfrac{1}{\eps}\mathcal{L} \{\mathbf{I}-\mathbf{P}\}\grad\p^\alpha_x g_\eps)_{L^2(\mathbb{R}^3_v)} \chi(x)\,\mathrm{d}x\mathrm{d}t\\
    \lesssim & \frac{1}{\eps^2}\int^{t_2}_{t_1} \cD^2_N(g_\eps(t))\,\mathrm{d}t\,.
\end{aligned}
\end{equation}
Thus the energy dissipation estimate \eqref{global bound D} implies the equi-continuity in $t$. From the Arzel\`{a}-Ascoli Theorem, $\tilde{\theta} \in C([0,\infty);H^{N-1-\eta}(\mathbb{R}^3_x))\cap L^\infty([0,\infty);H^{N-\eta}(\mathbb{R}^3_x))$, and
\begin{equation}\label{theta convergence}
   \tfrac{3}{5}\theta_\eps - \tfrac{2}{5}\rho_\eps \rightarrow \tilde{\theta}\quad \text{in}\quad\! C([0,\infty);H^{N-1-\eta}(\mathbb{R}^3_x))\cap L^\infty([0,\infty);H^{N-\eta}(\mathbb{R}^3_x))\,,
\end{equation}
as $\eps\rightarrow 0$ for any $\eta >0$. Note that $\tilde{\theta} = \tfrac{3}{5}\theta - \tfrac{2}{5}\rho$ and $\theta = (\tfrac{3}{5}\theta - \tfrac{2}{5}\rho) + \frac{2}{5}(\rho + \theta)$, and the relation \eqref{Boussinesq}, we get $\tilde{\theta}=\theta$ and $\rho + \theta =0$.

\noindent {\bf Convergence of $\mathcal{P}u_\eps$:} Taking the Leray projection operator $\mathcal{P}$ on the second equation of \eqref{local cons-law} gives
\begin{equation}\label{Pu equation}
 \p_t \mathcal{P}\mathrm{u}_\eps  + \mathcal{P} \grad\!\cdot\! (\sqrt{\mu}\widehat{A}\,, \tfrac{1}{\eps}\mathcal{L} g_\eps)_{L^2(\mathbb{R}^3_v)} =0\,.
\end{equation}
Similar arguments as above deduce that there exists a divergence free $\tilde{\mathrm{u}} \in L^\infty([0,\infty;H^N(\mathbb{R}^3_x)))$, such that
\begin{equation}\label{u convergence}
   \mathcal{P}\mathrm{u}_\eps \rightarrow \tilde{\mathrm{u}} \quad \text{in}\quad\! C([0,\infty);H^{N-1-\eta}(\mathbb{R}^3_x))\cap L^\infty([0,\infty);H^{N-\eta}(\mathbb{R}^3_x))\,,
\end{equation}
as $\eps\rightarrow 0$ for any $\eta >0$. Note that $\tilde{\mathrm{u}} = \mathcal{P}\mathrm{u}$ and \eqref{incompressibility}, we have $\tilde{\mathrm{u}}=\mathrm{u}$.

Follow the standard calculations, (for example, \cite{BGL1}), the local conservations laws can be rewritten as
\begin{equation}\label{NS with error}
\begin{cases}
   \p_t \rho_\eps + \tfrac{1}{\eps} \grad\!\cdot\! \mathrm{u}_\eps = 0\,, \\
   \p_t \mathrm{u}_\eps + \tfrac{1}{\eps} \grad(\rho_\eps + \theta_\eps) + \grad\!\cdot\! (\mathrm{u}_\eps \otimes \mathrm{u}_\eps - \frac{|\mathrm{u}_\eps|^2}{3}I) = \nu\grad\!\cdot\!\Sigma(\mathrm{u}_\eps) + \grad\!\cdot\! R_{\eps,\mathrm{u}}\,,\\
   \p_t \theta_\eps + \tfrac{2}{3}\tfrac{1}{\eps} \grad\!\cdot\! \mathrm{u}_\eps + \grad\!\cdot\! (\mathrm{u}_\eps \theta_\eps) = \kappa \grad\!\cdot\![\grad \theta_\eps] + \grad\!\cdot\! R_{\eps,\theta}\,,
\end{cases}
\end{equation}
where $\Sigma(\mathrm{u})=\grad \mathrm{u} + \grad \mathrm{u}^T - \frac{2}{3}\grad\!\cdot\! \mathrm{u} I\,,$ and $R_{\eps,\mathrm{u}}, R_{\eps,\theta}$ have of the form
\begin{equation}\label{error}
\begin{aligned}
  &-\eps(\zeta(v), \p_t g_\eps)_{L^2_v{(\mathbb{R}^3)}} + (\zeta(v), v\!\cdot\!\grad\{\mathbf{I}-\mathbf{P}\} g_\eps)_{L^2{(\mathbb{R}^3_v)}} + (\zeta(v), \Gamma(\{\mathbf{I}-\mathbf{P}\} g_\eps,\{\mathbf{I}-\mathbf{P}\} g_\eps))_{L^2{(\mathbb{R}^3_v)}}\\
  & + (\zeta(v), \Gamma(\{\mathbf{I}-\mathbf{P}\}g_\eps, \mathbf{P} g_\eps))_{L^2{(\mathbb{R}^3_v)}}+ (\zeta(v), \Gamma(\mathbf{P} g_\eps,\{\mathbf{I}-\mathbf{P}\} g_\eps))_{L^2{(\mathbb{R}^3_v)}}\,.
\end{aligned}
\end{equation}
For $R_{\eps,u}$, take $\zeta(v)= \sqrt{\mu}\widehat{A}$, while for $R_{\eps,\theta}$, take $\zeta(v)= \sqrt{\mu}\widehat{B}$.

\noindent{\bf The equations of  $\theta$ and $\mathrm{u}$: } Decompose $\mathrm{u}_\eps = \mathcal{P}\mathrm{u}_\eps + \mathcal{Q}\mathrm{u}_\eps$, where $\mathcal{Q}= \grad \Delta^{-1}_{\!x}\grad\cdot$ is a gradient. Denote $\tilde{\theta}_\eps= \tfrac{3}{5}\theta_\eps - \tfrac{2}{5}\rho_\eps$. Then from \eqref{NS with error}, the following equation is satisfied in the sense of distributions:
\begin{equation}\nonumber
  \p_t \tilde{\theta}_\eps + \tfrac{3}{5}\grad\!\cdot\!(\mathcal{P}\mathrm{u}_\eps \tilde{\theta}_\eps) - \tfrac{3}{5}\kappa \Delta_{\! x}\tilde{\theta}_\eps = \grad\!\cdot\!\widetilde{R}_{\eps,\theta}\,,
\end{equation}
where
\begin{equation}\label{R-theta}
  \widetilde{R}_{\eps,\theta} = \tfrac{3}{5}R_{\eps,\theta} - \tfrac{6}{25}\mathcal{P}\mathrm{u}_\eps (\rho_\eps + \theta_\eps)- \tfrac{6}{25}\mathcal{Q}\mathrm{u}_\eps (\rho_\eps + \theta_\eps)- \tfrac{3}{5}\mathcal{Q}\mathrm{u}_\eps \tilde{\theta}_\eps + \tfrac{6}{25} \kappa \Delta_{\!x} (\rho_\eps + \theta_\eps)\,.
\end{equation}

For any $T>0$, let $\phi(t,x)$ be a text function satisfying $\phi(t,x) \in C^1([0,T], C^\infty_c(\mathbb{R}^3_x))$ with $\phi(0,x)=1$ and $\phi(t,x)= 0$ for $t \geq T'$, where $T' < T$. Noting \eqref{error}, and using the global bounds \eqref{global bound E}, \eqref{global bound D} and \eqref{global bound C}, it is easy to show that
\begin{equation}\label{error vanish}
\int^T_0 \int_{\mathbb{R}^3_x} \grad\!\cdot\! R_\eps(t,x) \phi(t,x)\,\mathrm{d}x\mathrm{d}t \rightarrow 0\quad \text{as}\quad\! \eps \rightarrow 0\,,
\end{equation}
where $R_\eps = R_{\eps,\mathrm{u}}$ or $R_{\eps,\theta}$. For other terms in \eqref{R-theta}, noting that the convergence \eqref{div u vanish} and \eqref{rho-theta vanish}, together with \eqref{error vanish}, we have
\begin{equation}\label{R-error vanish}
\int^T_0 \int_{\mathbb{R}^3_x} \grad\!\cdot\! \widetilde{R}_{\eps,\theta}(t,x) \phi(t,x)\,\mathrm{d}x\mathrm{d}t \rightarrow 0\quad \text{as}\quad\! \eps \rightarrow 0\,.
\end{equation}

From the convergence \eqref{theta convergence} and \eqref{u convergence}, for $N >1$,  as $\eps \rightarrow 0$,
\begin{equation}\label{t theta }
 \int^T_0\int_{\mathbb{R}^3_x} \p_t \tilde{\theta}_\eps(t,x) \phi(t,x)\,\mathrm{d}x\mathrm{d}t \rightarrow -\int_{\mathbb{R}^3_x} (\tfrac{3}{5}\theta_0 - \tfrac{2}{5}\rho_0)(x)\mathrm{d}x - \int^T_0\int_{\mathbb{R}^3_x} \theta(t,x) \p_t \phi(t,x)\,\mathrm{d}x\mathrm{d}t\,,
\end{equation}
\begin{equation}\label{laplace theta}
  \int^T_0\int_{\mathbb{R}^3_x} \Delta_{\!x} \tilde{\theta}_\eps \phi(t,x)\,\mathrm{d}x\mathrm{d}t \rightarrow \int^T_0\int_{\mathbb{R}^3_x} \theta(t,x) \Delta_{\!x}\phi(t,x)\,\mathrm{d}x\mathrm{d}t\,,
\end{equation}
and
\begin{equation}\label{u-theta}
     \int^T_0\int_{\mathbb{R}^3_x} \grad\!\cdot\!(\mathcal{P}\mathrm{u}_\eps \tilde{\theta}_\eps)\phi(t,x)\,\mathrm{d}x\mathrm{d}t  \rightarrow -\int^T_0\int_{\mathbb{R}^3_x} \mathrm{u}(t,x) \theta(t,x)\!\cdot\! \grad \phi(t,x) \,\mathrm{d}x\mathrm{d}t\,.
\end{equation}

Acting the Leray projection $\mathcal{P}$ on the second equation of \eqref{NS with error}, we have the following equation
\begin{equation}\nonumber
  \p_t \mathcal{P}\mathrm{u}_\eps + \mathcal{P}\grad\!\cdot\!(\mathcal{P}\mathrm{u}_\eps \otimes \mathcal{P}\mathrm{u}_\eps)- \nu \Delta_{\!x}\mathcal{P} \mathrm{u}_\eps = \mathcal{P}\grad\!\cdot\! \widetilde{R}_{\eps,\mathrm{u}}
\end{equation}
where
\begin{equation}
\widetilde{R}_{\eps,\mathrm{u}} = R_{\eps,\mathrm{u}} - \mathcal{P}\!\cdot\!(\mathcal{P}\mathrm{u}_\eps \otimes \mathcal{Q}\mathrm{u}_\eps+\mathcal{Q}\mathrm{u}_\eps \otimes \mathcal{P}\mathrm{u}_\eps+ \mathcal{Q}\mathrm{u}_\eps \otimes \mathcal{Q}\mathrm{u}_\eps )\,.
\end{equation}

Similar as above we can take the vector-valued test function $\psi(t,x)$ with $\grad\!\cdot\! \psi=0$, and prove that as $\eps \rightarrow 0$
\begin{equation}
\begin{aligned}
  &\int^T_0\int_{\mathbb{R}^3_x}(\p_t \mathcal{P}\mathrm{u}_\eps + \mathcal{P}\grad\!\cdot\!(\mathcal{P}\mathrm{u}_\eps \otimes \mathcal{P}\mathrm{u}_\eps)- \nu \Delta_{\!x}\mathcal{P} \mathrm{u}_\eps)\cdot \psi(t,x)\,\mathrm{d}x\mathrm{d}t \\
  \rightarrow & -\int_{\mathbb{R}^3_x} \mathcal{P}\mathrm{u}_0(x)\!\cdot\! \psi(0,x)\,\mathrm{d}x - \int^T_0\int_{\mathbb{R}^3_x}(u\!\cdot\!\p_t\psi+ \mathrm{u}\otimes \mathrm{u}:\grad\psi-\nu \mathrm{u}\!\cdot\! \Delta_{\!x}\psi\,\mathrm{d}x\mathrm{d}t\,,
\end{aligned}
\end{equation}
and
\begin{equation}
\int^T_0 \int_{\mathbb{R}^3_x} \mathcal{P}\grad\!\cdot\! \widetilde{R}_{\eps,\mathrm{u}}(t,x) \phi(t,x)\,\mathrm{d}x\mathrm{d}t \rightarrow 0\quad \text{as}\quad\! \eps \rightarrow 0\,.
\end{equation}

Collecting all above convergence results, we have shown that $(\mathrm{u},\theta)\in C([0,\infty);H^{N-1}(\mathbb{R}^3_x))\cap L^\infty([0,\infty);H^{N}(\mathbb{R}^3_x))$ satisfies the following incompressible Navier-Stokes equations
\begin{equation}\nonumber
\begin{cases}
   \p_t \mathrm{u} + \mathrm{u}\!\cdot\! \grad \mathrm{u} + \grad p = \nu \Delta_{\!x}\mathrm{u}\,,\\
   \grad\!\cdot\! \mathrm{u}=0\,,\\
   \p_t \theta + \mathrm{u}\!\cdot\! \grad \theta = \kappa \Delta_{\!x}\theta\,,
\end{cases}
\end{equation}
with initial data:
\begin{equation}\nonumber
   \mathrm{u}(0,x)= \mathcal{P}\mathrm{u}_0(x)\,, \quad \theta(0,x)= \tfrac{3}{3}\theta_0(x)-\tfrac{2}{5}\rho_0(x)\,.
\end{equation}

\bigbreak

\noindent{\bf Acknowledgment:}
Ning Jiang was supported by a grant from the National Natural Science Foundation of China under contract  No. 11171173, Chaojiang Xu was supported by a grant from the National Natural Science Foundation of China under contract  No. 11171261, and Huijiang Zhao was supported by two grants from the National Natural Science Foundation of China under contracts 10925103 and 11261160485 respectively. This work was also supported
by the Fundamental Research Funds for the Central Universities.


\begin{thebibliography}{99}


\bibitem{amuxy3} R. Alexandre, Y. Morimoto, S.  Ukai, C.-J.   Xu
and T. Yang, {\it
Global existence and full regularity of the Boltzmann equation without angular cutoff}.
Comm. Math. Phys. {\bf 304} (2011) 513-581.


\bibitem{amuxy4-2} R. Alexandre, Y. Morimoto, S. Ukai,,  C.-J. Xu
and T. Yang, {\it The Boltzmann equation without angular cutoff in the whole
space: I. Global existence for soft potential}. J. Funct. Anal. {\bf 262}
 (2012), 915-1010.


\bibitem{amuxy4-3} R. Alexandre, Y. Morimoto, S. Ukai,,  C.-J. Xu
and T. Yang, {\it The Boltzmann equation without angular cutoff in the whole
space: II.  Global existence for hard potentia}l.
Anal. Appl.(Singap.) {\bf 9} (2011), 113-134.


\bibitem{amuxy7}
R. Alexandre, Y. Morimoto, S. Ukai,,  C.-J. Xu and T. Yang,
{ \it Local existence with mild regularity for the Boltzmann equation}.
Kinet. Relat. Models  {\bf 6} (2013), no. 4, 1011-1041.

\bibitem{al-3}
R. Alexandre and C. Villani,  {\it  On the Boltzmann equation for
long-range interaction}.  Commun. Pure and Appl. Math. {\bf 55} (2002),  30-70.

\bibitem{Arsenio} D. Arsenio,
{\it From Boltzmann's Equation to the incompressible Navier-Stokes-Fourier system with long-range interactions}. { Arch. Ration. Mech. Anal.} {\bf 206} (2012), no. 3, 367-488

\bibitem{BGL1} C. Bardos, F. Golse, and C. D. Levermore,
{\em Fluid dynamic limits of kinetic equations I: formal
derivation.,} J. Stat. Phys. {\bf 63} (1991), 323-344.

\bibitem {BGL2} C. Bardos, F. Golse, and C. D. Levermore,
{\em Fluid dynamic limits of kinetic equations II: convergence proof
for the Boltzmann equation.} Commun. Pure and Appl. Math. {\bf 46}
(1993), 667-753

\bibitem {BGL3} C. Bardos, F. Golse, and C. D. Levermore,
{\em The acoustic limit for the Boltzmann equation.} { Arch. Ration.
Mech. Anal.} {\bf 153} (2000), no. 3, 177-204

\bibitem{b-u}
C. Bardos and S. Ukai, {\it The classical incompressible Navier-Stokes limit of the Boltzmann equation.} Math. Models Methods Appl. Sci. {\bf 1} (1991), no.2, 235-257.

\bibitem{Caflisch}
R. Caflisch,
{\it The fluid dynamic limit of the nonlinear Boltzmann equation.} Comm. Pure Appl. Math. {\bf 33} (1980), no. 5, 651-666.

\bibitem{DEL-89}
A. De Masi, R. Esposito, and J. L. Lebowitz,
{\it Incompressible Navier-Stokes and Euler limits of the Boltzmann equation.}
 Comm. Pure Appl. Math. {\bf 42} (1989), no. 8, 1189-1214.

\bibitem{D-L} R. J. DiPerna, P.L.  Lions, {\it On the Cauchy problem for Boltzmann
equations: global existence and weak stability}.  Ann. Math.
{\bf 130}(1989), 321-366.

\bibitem{GL} F. Golse and C. D. Levermore,
{\em The Stokes-Fourier and acoustic limits for the Boltzmann
equation.} Comm. on Pure and Appl. Math. {\bf 55} (2002), 336-393.

\bibitem{Go-Sai04} F. Golse and L. Saint-Raymond,
{\em The Navier-Stokes limit of the Boltzmann equation for bounded
collision kernels.} Invent. Math. {\bf 155} (2004), no. 1, 81--161.

\bibitem{Go-Sai09} F. Golse and L. Saint-Raymond,
{\em The incompressible Navier-Stokes limit of the Boltzmann
equation for hard cutoff potentials.}  J. Math. Pures Appl. (9) {\bf
91} (2009), no. 5, 508--552.

\bibitem{G-S} P. Gressman and R. Strain,
{\it Global classical solutions of the Boltzmann equation without angular cut-off.}
{ J. Amer. Math. Soc.} {\bf 24} (2011), no. 3, 771-847


\bibitem{guo-T} Y. Guo,
{\it Classical solution to the Boltzmann Equation for molecules with an angular cutoff.}  Arch. Rational Mech. Anal. {\bf 169} (2003) 305-353.

\bibitem{guo-1} Y. Guo,
{\it The Boltzmann equation in the whole space.} Indiana Univ. Math. J. {\bf 53-4} (2004) 1081-1094.

\bibitem{GY-06} Y. Guo,
{\it Boltzmann diffusive limit beyond the Navier-Stokes approximation.} { Comm. Pure Appl. Math.} {\bf 59} (2006),  626-687.

\bibitem{KMN} S. Kawashima, A. Matsumura and T. Nishida,
{\em On the fluid dynamical approximation to the Boltzmann equation at the level of the Navier-Stokes equation,} Comm. Math. Phys., {\bf 70} (1979), 97-124.

\bibitem{LM} C. D. Levermore and N. Masmoudi,
{\em From the Boltzmann equation to an incompressible Navier-Stokes-Fourier system.}
Arch. Ration. Mech. Anal. {\bf 196} (2010), no. 3, 753-809.

\bibitem{LM3} P. L. Lions and N. Masmoudi,
{\em From Boltzmann equation to Navier-Stokes and Euler equations
I.} Arch. Ration. Mech. Anal. {\bf 158} (2001), 173-193.

\bibitem{LM4} P. L. Lions and N. Masmoudi,
{\em From Boltzmann equation to Navier-Stokes and Euler equations
II.} Arch. Ration. Mech. Anal. {\bf 158} (2001), 195-211.

\bibitem{liu-2}
T.-P. Liu, T. Yang, T., and S.-H. Yu, {\it
Energy method for Boltzmann equation},
{ Phys. D } {\bf 188}  (2004), 178-192.

\bibitem{Nishida} T. Nishida,
{\em Fluid dynamical limit of the nonlinear Boltzmann equation to the level of the compressible Euler equation,} Comm. Math. Phys., {\bf 61} (1978), 119-148.

\end{thebibliography}
\end{document}